\documentclass[11pt]{article}
\usepackage{amsmath}
\usepackage{mathrsfs}
\usepackage{amsfonts}
\usepackage{amsmath,amsthm}
\usepackage{amsmath,amssymb,amsthm,latexsym}
\usepackage{graphics}
\usepackage{subfigure}
\usepackage{float,caption}
\usepackage{graphicx,arcs}
\usepackage{amscd}
\usepackage[all]{xy}
\usepackage{hyperref}

\newtheorem{theorem}{Theorem}[section]
\newtheorem{proposition}[theorem]{Proposition}

\newtheorem{definition}[theorem]{Definition}
\newtheorem{corollary}[theorem]{Corollary}
\newtheorem{lemma}[theorem]{Lemma}

\newtheorem{problem}{Problem}
\theoremstyle{remark}
\newtheorem{remark}[theorem]{Remark}

\def\<{\langle}
\def\>{\rangle}

\begin{document}
\title{\bf{Mean Curvature Rigidity and Non-rigidity Results on Spherical Caps}}

\author{Shibing Chen%
  \thanks{School of Mathematical Sciences, University of Science and Technology of
China, Hefei 230026, P.R. China. \texttt{chenshib@ustc.edu.cn}. Research of Chen was supported by National Key R and D program of China 2020YFA0713100 and NSFC
grant 11201486..}
  \and
Xiang Ma%
  \thanks{LMAM, School of Mathematical Sciences, Peking University,
 Beijing 100871, P.R. China. \texttt{maxiang@math.pku.edu.cn}, Fax:+86-010-62751801. Corresponding author. Supported by NSFC grant 11831005.}
\and
 Shengyang Wang%
  \thanks{School of Mathematical Sciences, Peking University,
 Beijing 100871, People's Republic of China. \texttt{1900010752@pku.edu.cn}. } }

\date{\today}
\maketitle

\begin{center}
{\bf Abstract}
\end{center}

\hspace{2mm}
We prove that a hemisphere in the Euclidean space $R^{n+1}$, viewed as the graph of a function, admits no smooth perturbations as graphs with mean curvature $H\ge 1$ whose boundary equator is fixed up to $C^2$. This is an extension of the \emph{Mean Curvature Rigidity} phenomenon discovered by Gromov and Souam on non-compact totally umbilic hypersurfaces in space forms. The proof uses a Tangency Principle. On the other hand, we show that there exist nontrivial smooth perturbations with $H\ge 1$ on a great spherical cap whose boundary is fixed up to $C^2$. Similar results hold true for perturbations decreasing $H$, and for the $r$ mean curvature function $H_r$. This contrast between rigidity and non-rigidity is even true in the 1-dimensional case for circles and for discrete objects (polygons inscribed in a circle).\\

{\bf Keywords:}  spheres, spherical caps, mean curvature, Gauss-Kronecker curvature, rigidity theorem, Menger curvature \\


\section{Introduction}

In \cite{Gromov}, Gromov pointed out the following simple and elegant result:

\begin{theorem}\label{thm-gromov}[Gromov 2019]
A hyperplane $M$ in a Euclidean space $\mathbb{R}^{n+1}$ cannot be perturbed on a compact set so
that the perturbed hypersurface $\Sigma$ has mean curvature $H_{\Sigma} \ge 0$ unless $H_{\Sigma}\equiv 0$ and $\Sigma=M$ identically.
\end{theorem}

This theorem originated from Gromov's deep study of manifolds with non-negative scalar curvature
(which can be traced back further to the Positive Mass Theorem). Very soon, Souam \cite{Souam} gave an \emph{elementary} proof of this theorem. Moreover, Souam generalized Gromov's result to various noncompact \emph{spheres} in the hyperbolic space as below.

\begin{theorem}\label{thm-souam}[Souam 2021]
Let $M$ denote a horosphere, a hypersphere or a hyperplane in a hyperbolic
space $\mathbb{H}^{n+1}, n\ge 2$ and $H_M \ge 0$ be its constant mean curvature. Let $\Sigma$ be a connected properly embedded $\mathcal{C}^2$-hypersurface in $\mathbb{H}^{n+1}$ which coincides with $M$ outside a compact subset of $\mathbb{H}^{n+1}$. If the mean curvature of $\Sigma$ satisifes $H_{\Sigma}\ge H_M$, then $\Sigma = M$.
\end{theorem}

Souam's proof is elementary in the sense that it invokes only the classical technique called the \emph{moving plane method}, which dates back to E. Hopf. Roughly speaking, it comes from the following picture and intuition.\\

\noindent
\textbf{The Tangency Principle.}

Let $M_1$ and $M_2$ be two embedded $\mathcal{C}^2$-hypersurfaces in a Riemannian manifold $N$ with or without boundary. Suppose that they are tangent at a point $p$, where they share the same normal vector $\eta_0$. Locally they can be viewed as graphs over the same tangent plane $W$ and suppose $M_1$ is \emph{above} $M_2$ (denoted as $M_1\ge M_2$) with respect to the \emph{upwards} direction of $\eta_0$.
If the inequality is reversed and their $r$-mean curvatures satisfy $H_r(M_1) \le \alpha \le H_r(M_2)$ in a neighborhood, then in many quite natural cases, these conditions force $M_1=M_2$ in a neighborhood of $p$. Such criterions are known as the Tangency Principle.\\

This principle, which essentially follows from the maximum principle of elliptic PDE, has various versions and generalizations, and has already been widely used in the study of minimal and CMC surfaces. For detailed explanation and the general version, we refer to F. Fontenele and Sergio L. Silva's work \cite{Fontenele}, see Section~2. Souam's Tangency Principle \cite{Souam} is a special case for the mean curvature $H$. With this stronger Tangency Principle, the mean curvature rigidity results of Gromov and Souam can be extended to the $r$-mean curvature function $H_r$.

These results motivated our curiosity: Can these rigidity results be generalized to other curvature functions? (The answer is YES for $H_r$ by the Tangency Principle in \cite{Fontenele}.) Can we find similar rigidity phenomenon for other hypersurfaces like the minimal hypersurfaces, round cylinders, and more? In particular, can we generalize them to the compact ones, like the standard sphere $S^n\subset \mathbb{R}^{n+1}$?

The last problem seems quite naive, since one can increase (or decrease) the curvature of the round sphere by shrinking (or enlarging) its radius. To avoid such simple counter-examples to the desired curvature rigidity results, we need to make some restriction on the allowed perturbations.

In our project, different from Gromov and Souam's situation, here we are interested in smooth perturbations of a compact hypersurface $M$, yet with substantial part $B_0\subset M$ fixed. In particular, in this paper we only consider perturbing a cap of the sphere smoothly and increasing (or decreasing) the r-mean curvature $H_r$, while the complementary spherical cap is fixed (or equivalently, the boundary of the perturbed region is fixed up to $C^2$).
We find that the answer depends on the size of the perturbed region $B_0$. In this paper we make the following\\

\noindent
\textbf{Convention:}

  (1) The sphere $S^n\subset \mathbb{R}^{n+1}$ has radius $1$ (the unit sphere). Its mean curvature $H=1$ is the arithmetic mean of all the principal curvatures.

  (2) $S^+_{\theta}$ is a spherical cap with half central angle $\theta$ (intrinsically, a geodesic ball of $S^n$ centering at the north pole $N$ with radius $\theta$). It is a small spherical cap when $0<\theta<\pi/2 $, a great spherical cap when $\pi/2<\theta<\pi $, and a hemisphere when $\theta =\pi/2$. Denote its complementary part in $S^n$ as $S^-_{\theta}$, which is also a spherical cap with half central angle $\pi-\theta$.

  (3) A smooth perturbation $\Sigma$ of the spherical cap $S^+_{\theta}$ is a smooth hypersurface coincide with $S^n$ on $S^-_{\theta}$; in other words, we are considering the truly perturbed part as a deformation $\Sigma_0$ of $S^+_{\theta}$ which is compact (but $\Sigma_0$ is allowed to be of other topological type); $\Sigma_0$ and $S^+_{\theta}$ share the same boundary up to $\mathcal{C}^2$.

Our rigidity results can be summarized as below (see Section~4 for more general conclusions and proofs):

\begin{theorem}\label{thm-rigid-0}
Suppose $\Sigma$ is a smooth perturbation of the round sphere $S^n$ which coincide with $S^n$ on a hemisphere, and $\Sigma$ is convex with the r-mean curvature $H_r\ge 1$ everywhere (or $H_r\le 1$ everywhere). Then $\Sigma=S^n$.
\end{theorem}

If the perturbed region is a great spherical cap, then the situation is totally different. Now we can construct nontrivial perturbations increasing (or decreasing) $H_r$. In particular, we need only to consider hypersurfaces of revolution and we can control the principal curvatures easily.

\begin{theorem}\label{thm-nonrigid-0}
There exists a smooth hypersurface $\Sigma$ which coincides with $S^n$ on a small spherical cap such that $\Sigma$ is a convex hypersurface of revolution and its principal curvatures $k_i> 1, \forall i=1,2,\cdots, n$ outside this small spherical cap. Similarly, there exist nontrivial perturbations of the great spherical cap $S^+_{\theta}$ whose principal curvatures $0<k_i\le 1, \forall i=1,2,\cdots, n$.
\end{theorem}

This dramatic contrast seems quite unbelievable to us at the beginning. To verify these conclusions (at first only vague conjectures for us) and to have a better understanding, a heuristic way is to look at the most elementary case, namely a quadrilateral $ACBD$ inscribed in a circle as in the figure below. When we perturb the vertex $D$ (the major arc $\overset{\frown}{ADB}$) and fix $A,C,B$, the discrete \emph{Menger curvature} at $A,D,B$ can increase or decrease simultaneously (while that at $D$ is fixed); this is impossible if we fix $A,D,B$ and perturb only $C$ slightly. (For the definition of the Menger curvature and a \emph{discrete curvature rigidity theorem}, see the appendix at Section~7.)

 \begin{figure}[!h]
   \setlength{\belowcaptionskip}{0.1cm}
   \centering
   \includegraphics[width=0.4\textwidth]{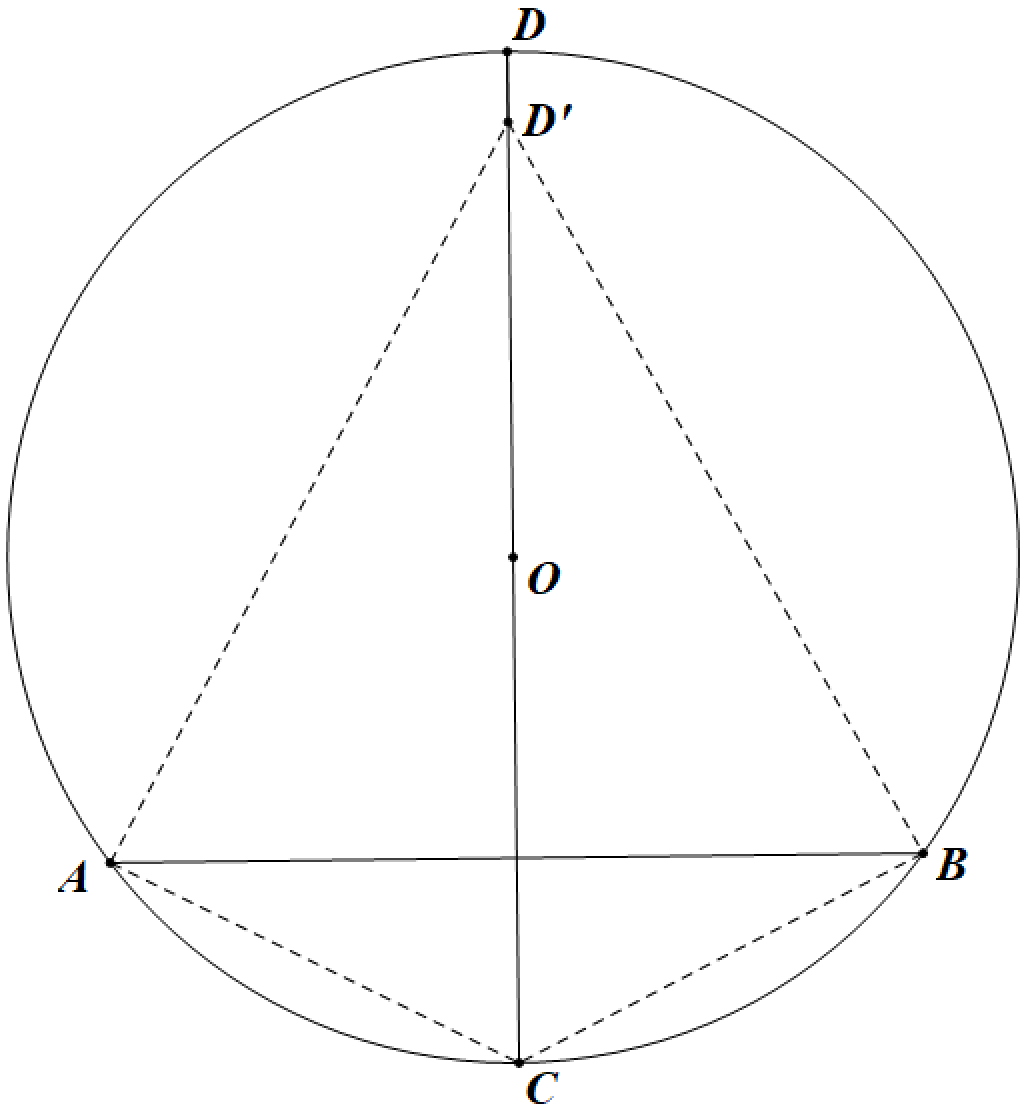}
   \caption*{Perturb the quadrilateral inscribed in a circle}
   \label{fig1}
 \end{figure}

How to understand the difference between the small and great spherical caps? From the proof of the mean curvature rigidity theorem we will see, the crucial difference is that when the spherical cap $S_{\theta}$ is small ($\theta\le \pi/2$), there exists a one-parameter family of spherical caps with the same boundary which has smaller curvature inside $S_{\theta}$ and larger curvature outside $S_{\theta}$. This phenomenon contrasts with the intuition justified by the Tangency Principle. On the other hand, for a great spherical cap $S^+_{\theta}$, the inside caps with the same boundary do have smaller radius (and larger curvature), which allow more freedom to deform.

More generally, the common feature underlying these curvature rigidity results is the existence of such a family of comparison (hyper)-surfaces $\{F_t\}$ (\emph{the slice}) inside a suitable chosen domain (\emph{the trap}). Extending this observation, this recipe is formulated as the so-called \emph{trap-slice lemma} (Theorem~\ref{thm-trap-slice}). Omitting the technical assumptions, it can be summarized in one sentence:\\

\noindent
\textbf{The trap-slice lemma.}
\emph{With respect to the outward normal direction, the boundary region $B_0\subset \partial\Omega$ of domain $\Omega$ admits no inward deformation $\Sigma_*$ with small $H_r(\Sigma_*)\le \alpha$ when $\Omega$ is foliated by $\{F_t\}$ with bigger $H_r(F_t)\ge \alpha$ and $\partial F_t\subset \partial\Omega\setminus B_0$. }\\

Our trap-slice lemma is essentially the same method as Souam used in \cite{Souam}. The advantage of this universal version is that it unifies similar \emph{mean curvature rigidity} results as well as their proofs. The study of such rigidity problem is thus determined by whether we can find suitable trap with desired slice (many times we use CMC-foliation).

We should mention that there has been a lot of research on the metric rigidity problem on the sphere and hemisphere (see \cite{Brendle, Gromov, HangFB-06} and references therein). Many of them originated from \\

\noindent
\textbf{Min-Oo's conjecture \cite{MinOo} in 1995}. \par
\emph{Let $(M,g)$ be an $n$-dimensional compact Riemannian manifold with boundary and the scalar curvature
$R\ge n(n-1)$. The boundary is isometric to the standard sphere $S^{n-1}$ and is totally geodesic. Then, $(M,g)$ is isometric to the hemisphere $S^+$.}\\

This conjecture has been verified in many special cases, yet refuted finally by Brendle, Marques and Neves in 2011 \cite{Brendle}.

Although that, Fengbo Hang and Xiaodong Wang \cite{HangFB-06} proved in 2006 that Min-Oo's conjecture under the additional condition that the metric $g$ is conformal to the standard metric of $S^+$. On the other hand, on a great spherical cap (a geodesic ball with radius $>\pi/2$) they showed in the same paper the existence of a nontrivial conformal metric with $R\ge n(n-1)$ and the metric agrees with the original metric in a neighborhood of $\partial S^+$.

Compared to Hang and Wang's non-rigidity result, we shows that such perturbation (not necessarily conformal) can be constructed directly using the ambient geometry, which seems stronger. On the other hand, they had another result which says that one can intrinsically perturb the sphere metric in an arbitrarily small $\epsilon$ neighborhood of the equator so that the sectional curvature increase, which by our discovery, is impossible for the ambient perturbation when $\epsilon$ is too small (see the discussion in Subsection~6.3 and our paper in preparation \cite{MaWang}).   \\

This paper is organized as follows. In Section~2 we present the Tangency Principle for $r$-mean curvature. It implies the trap-slice lemma in Section~3. As the application, we obtain several rigidity theorems for small spherical caps and hemispheres in Section~4.
In section~5 we give the construction of non-trivial perturbations of any great spherical cap using a specific deformation of the profile curve $\Gamma$, where (instead of any analytical method) we will take use of the famous correspondence between involute and evolute. All our methods are essentially geometric and classical.

\section{A Tangency Principle}

In this section, we present Fontenele and Silva's \textbf{Tangency Principle} following their convention in \cite{Fontenele}. Souam's tangency principle in \cite{Souam} can be regarded as a special case of this general version.

Let $(N^{n+1},h)$ be a complete Riemannian manifold with metric $h$ and the usual exponential mapping $\rm{exp}: T N \to N$. $M^n$ is a hypersurface of $N^{n+1}$. Given $p\in M^n$ and a fixed unit normal
vector $\eta_0\perp T_p M$, we parametrize a neighborhood of
$p\in M^n\subset N^{n+1}$ as the graph of a function $\mu=\mu(x)$:
\begin{equation}\label{eq-exp}
\phi(x) = \rm{exp}_p (x + \mu(x)\eta_0),
\end{equation}
where the tangent vector $x$ varies in a neighborhood $W$ of zero in $T_p M$ and
$\mu: W \to \mathbb{R}$ satisfies $\mu(0) = 0$.

Fix the local orientation by a normal vector field
$\eta: W\to T^\bot M$ of $M^n$ with $\eta(0) = \eta_0$.
Then $A_{\eta(x)}$, the second fundamental form of $M^n$ in the direction $\eta(x)$,
has principal curvatures $\lambda_1(x) \le \lambda_2(x) \le \cdots \lambda_n(x)$ as
continuous functions on $W$. Denote by
$\Lambda(x) = (\lambda_1(x), \lambda_2(x),\cdots, \lambda_n(x))$
the principal curvature vector at $x\in W$.
The $r$-mean curvatures $H_r, 1 \le r \le n$, is given by
\[
H_r(x) = \frac{1}{\begin{pmatrix} n \\ r \end{pmatrix}}\sigma_r(\lambda(x)),
\]
where $\sigma_r(\Lambda(x))$ is the value at $\Lambda(x)$ of the $r$-elementary symmetric function.
In particular, $H_1$ is the usual mean curvature $H$, and $H_n$ is usually called the \emph{Gauss-Kronecker curvature} for a hypersurface in $\mathbb{R}^{n+1}$.

Let $\Gamma_{r}$ denote the connected component in $\mathbb{R}^n$ of the set $\{ \sigma_r > 0 \}$ that
contains the vector $a_0 = (1,1, \cdots, 1)$. Observe that $\Gamma_n$ is precisely the positive
cone ${\cal O}^n$, defined by
\[ {\cal O}^n = \{(z_1, z_2, \cdots, z_n)\in \mathbb{R}^n, | z_i > 0,~~~ 1 \le i \le n\}.\]
and that ${\cal O}^n\subset \Gamma_{r}$ for $1\le r\le n$. More generally, $\Gamma_{r+1}\subset \Gamma_{r}$ for $1 \le r \le n-1$ \cite{Fontenele}.

It is usually called \emph{$r$-convex} if $\Lambda(p)$ (the principal curvature vector of $M$ at $p$) belongs to $\Gamma_r$, which amounts to requiring that at $p\in M$, $H_1,\cdots, H_r$ are all positive.

\begin{definition}\label{def-above}
Let $M^n_1$ and $M^n_2$ be two hypersurfaces of $N^{n+1}$ that are tangent
at $p$, i.e., they share the same tangent plane $T_p M_1 = T_p M_2$.
Fix a unit vector $\eta_0$ that is normal to $M^n_1$ at $p$.

We say that $M_1$ remains above $M_2$ in a neighborhood of $p$
with respect to $\eta_0$ if, when we parametrize $M_1$nd $M_2$ by $\phi_1$ and $\phi_2$ as
above, the corresponding functions $\mu_1$ and $\mu_2$ satisfy $\mu_1\ge\mu_2$ in a
neighborhood of zero.

We denote this as $M_1\ge M_2$ for simplicity when the orientation given by $\eta_0$ is clear.
\end{definition}

\begin{theorem}\label{thm-TP1}
Let $M^n_1$ and $M^n_2$ be hypersurfaces of $N^{n+1}$ that are tangent at $p$ and let $\eta_0$ be a unit normal vector of $M_1$ at $p$. Denote by $H_r^i(x)$ the $r$-mean curvature at $x\in W$ of $M_i, i=1,2,$ respectively. Suppose that with respect to this given $\eta_0$, we have:
\begin{enumerate}
 \item Locally $M_1\ge M_2$, i.e., $M_1$ remains above $M_2$ in a neighborhood of $p$;
 \item $H_r^2(x)\ge H_r^1(x)$ in a neighborhood of zero for some $r, 1 \le r \le n$; if $r\ge 2$, assume also that $M_2$ is $r$-convex at $p$.
\end{enumerate}
Then $M_1$ and $M_2$ coincide in a neighborhood of $p$.
\end{theorem}

Such a principle holds true mainly due to the ellipticity of the differential operator involved and the maximal principle. This is also why we require the principal curvature vector $\Lambda$ belongs to $\Gamma_r$ when $r\ge 2$. For the proof and discussions, see \cite{Fontenele}.

For hypersurfaces with boundaries, this tangency principle is still valid as pointed out in \cite{Fontenele}. One needs only to add that $M_1$ and $M_2$, as well as $\partial M_1$ and $\partial M_2$, are tangent at $p\in \partial M_1\cap \partial M_2$, and the statement of other parts is the same.

\section{The trap-slice lemma}

To apply the Tangency Principle to our rigidity problem, a typical recipe is to construct a trapping region $\Omega$ containing the possible perturbation $\Sigma_0$, and $\Omega$ is foliated by a family of smooth hypersurfaces $F_t$ whose $r$-mean curvature $H_r(F_t)$ is controlled and compared with $H_r(\Sigma)$. Usually we find an extremal hypersurface $F_T$ tangent to $\Sigma$ at an interior point $p$ and obtain contradiction with the Tangency Principle. The statement is as below. It seems somewhat too technical. The reward is that we will have a rigidity theorem and a method general enough; and the possible difficulties occurring at $\partial\Sigma_0$ as well as on $\partial F_t$ is treated seriously in our proof (which is easy to overlook at first sight).

\begin{theorem}\label{thm-trap-slice}[The trap-slice lemma]
Let $N^{n+1}$ be a Riemannian manifold.

The \emph{trap} $\Omega\subset N$ is a domain enclosed by two connected hypersurfaces $B_0,B_1$ sharing a boundary $A=B_0\cap B_1$ and $\partial\Omega=B_0\cup B_1$.

The \emph{slice} is a foliation of $\Omega$ by a one-parameter family of hypersurfaces $\{F_{t}\}\subset \Omega$ (with or without boundaries). When $\partial F_t\ne \emptyset$, we assume $\partial F_t \subset B_1$. Each $F_{t}$ divides $\Omega$ into two sub-domains, one having $B_0$ on its boundary, and $\Omega_t$ is the other one away from $B_0$.

Fix an integer $r$, $1\le r\le n$, and a real constant $\alpha\in\mathbb{R}$.
With respect to the outward normal of $\partial\Omega_t\supset F_t$, suppose that the $r$-mean curvature function of $F_{t}$ always satisfies $H_r(F_{t})\ge \alpha$, and each $F_t$ is $r$-convex when $r\ge 2$.

Given the trap and the slice as above, there does NOT exist any hypersurface $\Sigma_*$ with boundary $\partial\Sigma_*$ satisfying all of the following conditions:
\begin{enumerate}
  \item $\Sigma_*$, the interior of the compact hypersurface $\overline\Sigma_*=\Sigma_*\cup \partial\Sigma_*$, is embedded in $\Omega$ with boundary $\partial\Sigma_*\subset B_0\subset \partial\Omega$. In particular, $\Sigma_*$ divides $\Omega$ into two sub-domains; sub-domain $\Omega_*$ is the one of them that having $B_1$ on its boundary. We orient $\Sigma_*$ by the outward normal of $\partial\Omega_*$.
  \item The boundary $\partial\Sigma_*$ has a neighborhood $U_t$ in $\overline\Sigma_*$ not contained in $\Omega_t$ for any $t$.
  \item Given the orientation of $\Sigma_*$, the $r$-mean curvature function $H_r(\Sigma_*)\le \alpha$.
\end{enumerate}
\end{theorem}

\begin{proof}
We prove by contradiction. Suppose $\Sigma_*\subset \Omega$ is such a hypersurface and assume it to be connected (otherwise we choose any of its connected components).

Without loss of generality, choose the parameter $t$ so that the sub-domains $\Omega_{t_0}\subset \Omega_{t_1}$ when $t_0\ge t_1$. Intuitively, the leaves $F_t$ moves away from $B_0$ when $t$ increases.

For any interior point $p\in \Sigma_*$, there is a unique $t$ such that $F_t$, one leave of the foliation, passes $p$. This defines a continuous function on $\Sigma_*$ and we denote it by $f$. Notice that $f$ is not defined on the boundary $\partial\Sigma_*$, since two leaves $F_{t_1}, F_{t_2}$ might intersect on the same boundary point $p\in \partial\Sigma_*$, which makes $f$ not well-defined.

Suppose $t_0=f(p_0)$ for some $p_0\in \Sigma_*$.
Consider the closed sub-domain $\overline\Omega_{t_0}$ (the closed closure of $\Omega_{t_0}$).
Clearly $\Sigma_*\cap\overline\Omega_{t_0}=f^{-1}([t_0,+\infty))$, and it is a nonempty closed subset of $\Sigma_*$.
By condition~2 we know $\partial\Sigma_*$ has an open neighborhood $U_{t_0}\subset \Sigma_*$
such that $U_{t_0}$ has no intersection with $\Omega_{t_0}$.
This implies that inside $\overline\Sigma_*$, the boundary $\partial\Sigma_*$ is separated from $\Sigma_*\cap\Omega_{t_0}$.
As a consequence, $\Sigma_*\cap\overline\Omega_{t_0}$ is also a closed subset of the compact hypersurface $\overline\Sigma_*$.
Thus $f$ can be viewed as a continuous function defined on this compact subset  $\Sigma_*\cap\overline\Omega_{t_0}$,
which must attain the maximum $T$ at interior points.

At any maximum point $p_0\in f^{-1}(T)$, $\Sigma_*$ must be
tangent to $F_T$ with the same normal direction $\eta_0$ determined by the outward direction of $\partial\Omega_T$ at $p_0$.
By the maximum property it is also obvious that, with respect to $\eta_0$ and its continuous extension to a local normal vector field, $\Sigma_*$ is above $F_T$ in a neighborhood of $p_0$. In other words, $\Sigma_*\ge F_T$.
On the other hand, by assumptions we know $H_r(F_T)\ge \alpha\ge H_r(\Sigma_*)$, and the principal curvature vector $\Lambda(F_T)$ belongs to $\Gamma_r$ when $r\ge 2$.
By the Tangency Principle above, $F_T$ and $\Sigma_*$ coincide in a small neighborhood of $p_0$.
Since $p_0$ is arbitrary in $f^{-1}(T)$ we conclude that $f^{-1}(T)$ is an open and closed subset of $\Sigma_*$.
Thus $\Sigma_*=F_T$ identically. Yet this contradicts with condition 2 that $\partial\Sigma_*$ should have a neighborhood separated from $\Omega_T$. This finishes the proof.
\end{proof}

\begin{corollary}\label{cor-trap-slice}
Assumptions on the trap $\Omega\subset N, \partial\Omega=B_0\cup B_1$ and the slice $\{F_t\}$
are as in the trap-slice lemma (Theorem~\ref{thm-trap-slice}).
Moreover, we suppose that:
\begin{enumerate}
  \item $B_0$ is also one leave of the foliation $\{F_t\}$ (we may suppose $B_0=F_0$ is an open subset of $\partial\Omega$);
  \item For any other $t\ne 0$, either $\partial B_0 \cap \partial F_t=\emptyset$, or $B_0$ intersects with $F_t$ at their boundaries transversally.
\end{enumerate}
Then $B_0$ admits no non-trivial perturbation $\Sigma_0$ (with fixed boundary up to $C^2$ and the same orientation on $\partial\Sigma=\partial B_0$) such that $H_r(\Sigma_0)\le \alpha$, unless two hypersurfaces $\Sigma_0$ and $B_1$ intersect at their interior points.
\end{corollary}

\begin{proof}
If $\Sigma_0$ is contained outside the open domain $\Omega$, then we apply the Tangency Principle (the version for hypersurfaces with boundary) to $\Sigma_0$ and $B_0$ with respect to the chosen outward normal direction, and the conclusion is that $\Sigma_0$ and $B_0$ must coincide on an open and closed subset of $\Sigma_0$. Hence $\Sigma_0=B_0$, which is only a trivial perturbation.

When $\Sigma_0\cup \Omega\ne \emptyset$, let us suppose $\Sigma_0$ and $B_1\subset \partial\Omega$ do not intersect at their interior points. We want to deduce contradiction from these assumptions.

Choose any connected component of $\Sigma_0\cup \Omega$ and denote it $\Sigma_*$. We will apply the trap-slice lemma to this $\Sigma_*$ with boundary $\partial\Sigma_*\in B_0$.

We verify condition~2 of the trap-slice lemma, i.e.:
\emph{fix any $t>0$ and the sub-domain $\Omega_t$,
the boundary $\partial\Sigma_*$ has a neighborhood $U_t$
in $\overline\Sigma_*$ not contained in $\Omega_t$.}

In case that $\partial\Sigma_*$ is contained in the interior of $B_0$,
then $\partial\Sigma_*$ is separated from the closed sub-domain $\overline\Omega_t$.
So condition~2 holds true.

If $\partial\Sigma_*\cap \partial B_0\ne \emptyset$, for any such common boundary point $p$,
due to the transversal intersection assumption, there is an open neighborhood $U(p)$ separating $p$ from $\Omega_t$. Since $\partial\Sigma_*$ is compact, it is clear that we can find such an open neighborhood $U_t$ for $\partial\Sigma_*$ separating it from $\Omega_t$.

In conclusion, condition~2 in the statement of the trap-slice lemma is true under our assumptions for $\Sigma_*$. Then this lemma implies that such a hypersurface $\Sigma_*$ does not exist in the open domain $\Omega$.

There remains the final possibility: $\Sigma_0$ is contained in the closed domain $\overline\Omega$ and $\Sigma_0$ is tangent to the boundary hypersurface $B_1$ in an interior point $p_0\in \Sigma_0$.
Since $B_1$ is now the smooth limit of the foliation $\{F_t\}$,
there must be $H_r(B_1)\ge \alpha\ge H_r(\Sigma_0)$ with respect to the \emph{inward normal direction of $B_1=\partial \Omega$}; when $r\ge 2$, taking limit on $t$ also implies that $\Lambda(B_1)$ still belongs to $\Gamma_r$.
By our convention, $\Sigma_0$ and $B_1$ are endowed with the same normal direction $\eta_0$ at $p_0$, and clearly $\Sigma_0$ remains above $B_1$.
The Tangency Principle now applies at here.
So $\Sigma_0$ coincide with $B_1$ on an open and closed non-empty subset, hence on the whole $\Sigma_0$. But this is impossible, because at the boundary
$B_0\cap B_1$, $\Sigma_0$ is assumed to have the same normal direction as $B_0$, which is the outward normal direction of $\partial\Omega$. (Intuitively, $B_1$ is not a perturbation of $B_0$ in our sense, because they share the same boundary yet with opposite orientations.)

The last contradiction verifies the conclusion and finishes the proof.
\end{proof}

\begin{remark}
In other words, under the assumptions of this corollary, any non-trivial perturbation of $B_0$ must break \emph{the wall of the trap} set by $B_1$.
\end{remark}

\begin{remark}\label{remark-rconvex}
The trap-slice lemma and the corollary above are still true when the assumptions are changed as below:
$\Sigma_*$ and $F_t$ are oriented by the inward normal vectors with respect to $\Omega_*$ and $\Omega_t$, respectively,
and the inequality on $H_r$ is reversed as
\[H_r(F_t)\le \alpha\le H_r(\Sigma_*).\]
But here we need $\Sigma_0$ to satisfy the $r$-convexity condition (not the slice $\{F_t\}$) so that the Tangency Principle still apply to this case.
\end{remark}

As an application and demonstration of the above trap-slice lemma, we obtain the following rigidity result for horospheres and hyperspheres with positive constant $r$-mean curvature $H_r(M)=\alpha>0$ in the hyperbolic space $\mathbb{H}^{n+1}(-1)$.

\begin{theorem}\label{thm-souam+}
Let $M$ denote a horosphere or a hypersphere in a hyperbolic
space $\mathbb{H}^{n+1}(-1), n\ge 2$ and $H_r(M)=\alpha>0 $ be its constant mean curvature.
Let $\Sigma$ be a connected properly embedded $\mathcal{C}^2$-hypersurface in $\mathbb{H}^{n+1}$ which coincides with $M$ outside a compact subset $B_0$ of $\mathbb{H}^{n+1}$. If the mean curvature of $\Sigma$ satisifes $H_r(\Sigma)\le H_r(M)$, then $\Sigma = M$.
\end{theorem}

\begin{proof}
$M$ divides the ambient hyperbolic space into two domains. One of them is concave; we denote it as $\Omega$ (\emph{the trap}), and fix the outward normal direction on $M$.
$\Omega$ is foliated by a family of totally umbilical hypersurfaces $\{F_t\}$ without boundary (or have their boundaries on $\partial_{\infty}\mathbb{H}^{n+1}$).
When $M$ is a horosphere, we choose $\{F_t\}$ to be the one-parameter family of horospheres tangent to $M$ at the same boundary point at infinity.
When $M$ is a hypersphere, we choose $F_t$ to be the one-parameter family of hyperspheres sharing the same boundary at infinity as $M$.

In any case, the normal vector field $\eta$ is chosen to point to the convex side (the outward direction of $\Omega_t$), for which we have $H_r(F_t)\ge H_r(M)=\alpha$. Also note that $B_0\subset M$ can be added to the family $\{F_t\}$ as one leave of the foliation. The conclusion follows directly from Corollary~\ref{cor-trap-slice} since $B_1$, the other part of the boundary $\partial\Omega$, is now at infinity.
\end{proof}

This result is quite similar to Souam's rigidity theorem \ref{thm-souam}; the only difference is that Souam considered perturbations increasing the mean curvature function, whereas we consider the other direction, namely perturbations with smaller $r$-mean curvature. This case was not considered in Souam's paper \cite{Souam}. Yet the method is essentially the same; what we did is merely rewriting the original proof in Souam's paper \ref{thm-souam} in terms of the trap-slice recipe.

\section{The curvature rigidity of a small spherical cap}

Let $S^n(1)$ denote the unit sphere in $\mathbb{R}^{n+1}$ with $n\ge 1$. $S^+_\theta\subset S^n(1)$ is a spherical cap defined to be
\[
S^+_\theta=\{(x_0,x_1,\cdots,x_n)\in \mathbb{R}^{n+1} |\Sigma_{i=0}^n (x_i)^2=1, \cos\theta\le x_0 \le 1\}
\]
with $\theta$ the central half angle, as in the introduction. In this section, most of the time we assume $\theta\le \pi/2$, i.e., it is a small spherical cap or a hemisphere. In particular, let $S^+$ and $S^-$ denote the upper and the lower hemisphere, respectively.

First let us consider perturbations of $S^+_\theta$ with $H_r\ge 1=H_r(S^n)$. This is simpler than the case for perturbations with smaller $H_r\le 1$.

\begin{theorem}\label{thm-rigid-1}
Let $\Sigma_0$ be a smooth perturbation in $\mathbb{R}^{n+1}$ of the small spherical cap $S^+_{\theta}$ (or hemisphere), and the boundary is fixed up to $C^2$. We assume that $\Sigma=\Sigma_0\cup S^-_{\theta}$ is still a properly embedded closed hypersurface. If the $r$-mean curvature of $\Sigma$ satisfies $H_r (\Sigma)\ge 1$ and $\Sigma_0$ is $r$-convex when $r\ge 2$, then $\Sigma = S^n$.
\end{theorem}
\begin{proof}
\textbf{Step 1.} We claim that $\Sigma_0$ has no points in the lens domain $\Omega^-$ enclosed by the hyperplane $\Pi$ with equation $x_0=\cos\theta$ and the lower spherical cap $S^-_{\theta}$.
Otherwise, suppose $\Sigma_0\cap \Omega^- =\Sigma_*$ is a hypersurface with boundary $\partial\Sigma_*\subset \Pi$.
Let $\Omega^-$ be the trap; the hyperplanes parallel to $\Pi$ give a natural slice of this trap, which have $H_r=0<1=H_r(\Sigma)$. By the trap-slice lemma, we find this is impossible.

\begin{figure}[!h]
 \setlength{\belowcaptionskip}{0.2cm}
  \centering
   \includegraphics[width=0.5\textwidth]{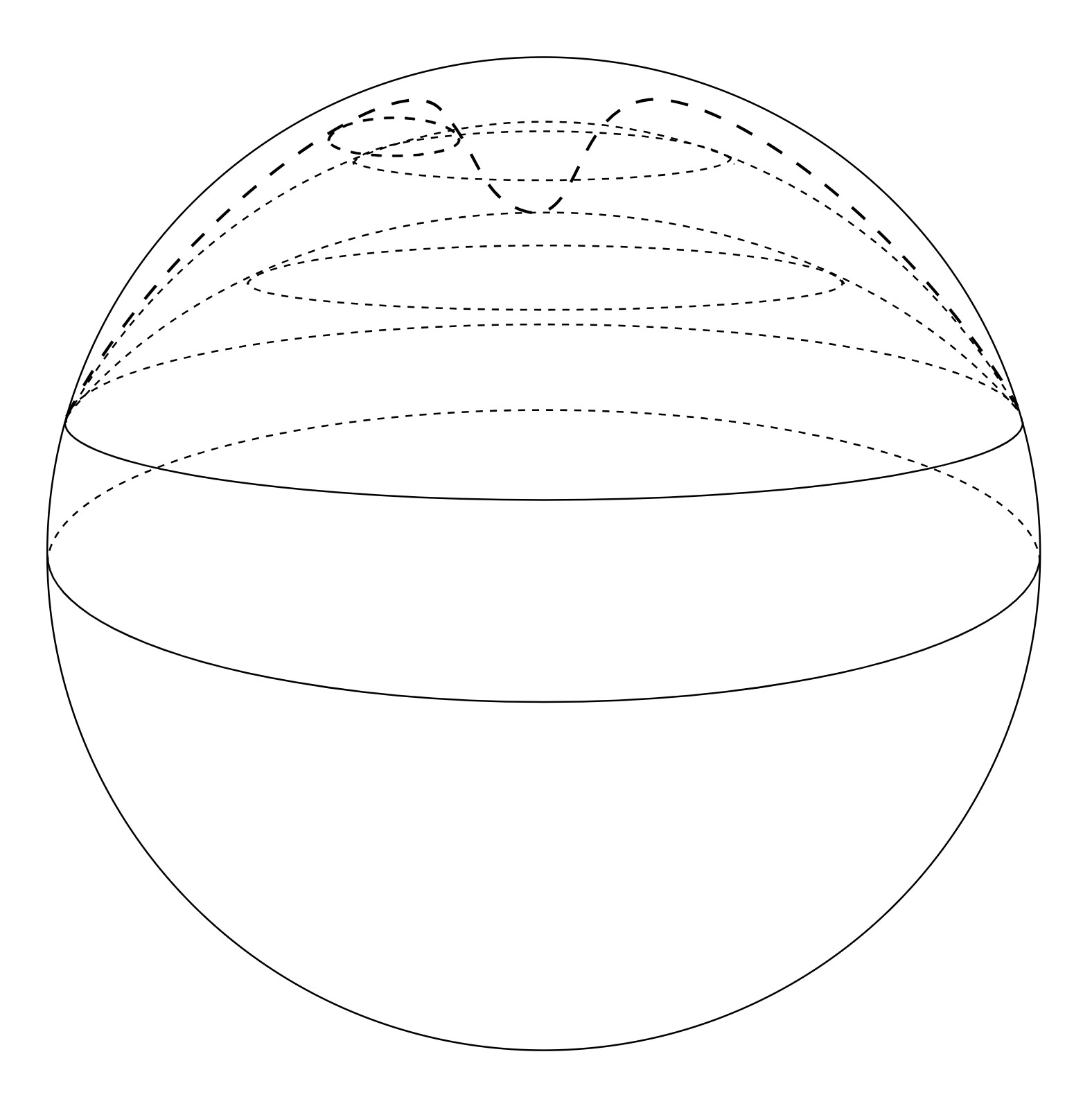}
   \label{fig1}
\end{figure}

\textbf{Step 2.} Then we consider the lens domain $\Omega^+$ enclosed by the hyperplane $\Pi$ and $S^+_{\theta}$. We choose $\Omega^+$ to be the new trap, which is sliced by the sphere caps in the unit ball passing through $\partial S^+_{\theta}$. Notice that $S^+_{\theta}$ is one leave of this family of spherical caps. By Corollary~\ref{cor-trap-slice} and Remark~\ref{remark-rconvex} we can still use the Tangency Principle and conclude with $\Sigma_0=S^+_{\theta}$.
\end{proof}

When we consider the perturbation in the other direction, namely $\Sigma$ has smaller $r$-mean curvature $H_r$,
the convexity condition can be weakened to graph conditions or that $\Sigma$ is properly embedded in a dumbbell-shaped domain.

\begin{theorem}\label{thm-rigid-2}
Let $\theta\le \pi/2$; $S^+_{\theta}$ is a hemisphere or a small spherial cap.
Let $\Sigma_0$ be a smooth perturbation of $S^+_{\theta}$ in $\mathbb{R}^{n+1}$
with the same boundary of $\partial S^+_{\theta}$ up to $C^2$.
In other words, $\Sigma=\Sigma_0\cup S^-_{\theta}$ is a closed hypersurface
which coincides with $S^n$ on a hemisphere (or a bigger cap).
Assume that $\Sigma_0$ is also a graph over the coordinate plane $x_0=0$ and $H_r(\Sigma)\le 1$.
Then $\Sigma_0=S^+_{\theta}, \Sigma=S^n$.
\end{theorem}
\begin{proof}
We will prove for the upper hemisphere $S^+$,
which implies rigidity results for any small spherical cap (a subset of the hemisphere) directly.

 \begin{figure}[!h]
   \setlength{\belowcaptionskip}{0.2cm}
   \centering
   \includegraphics[width=0.7\textwidth]{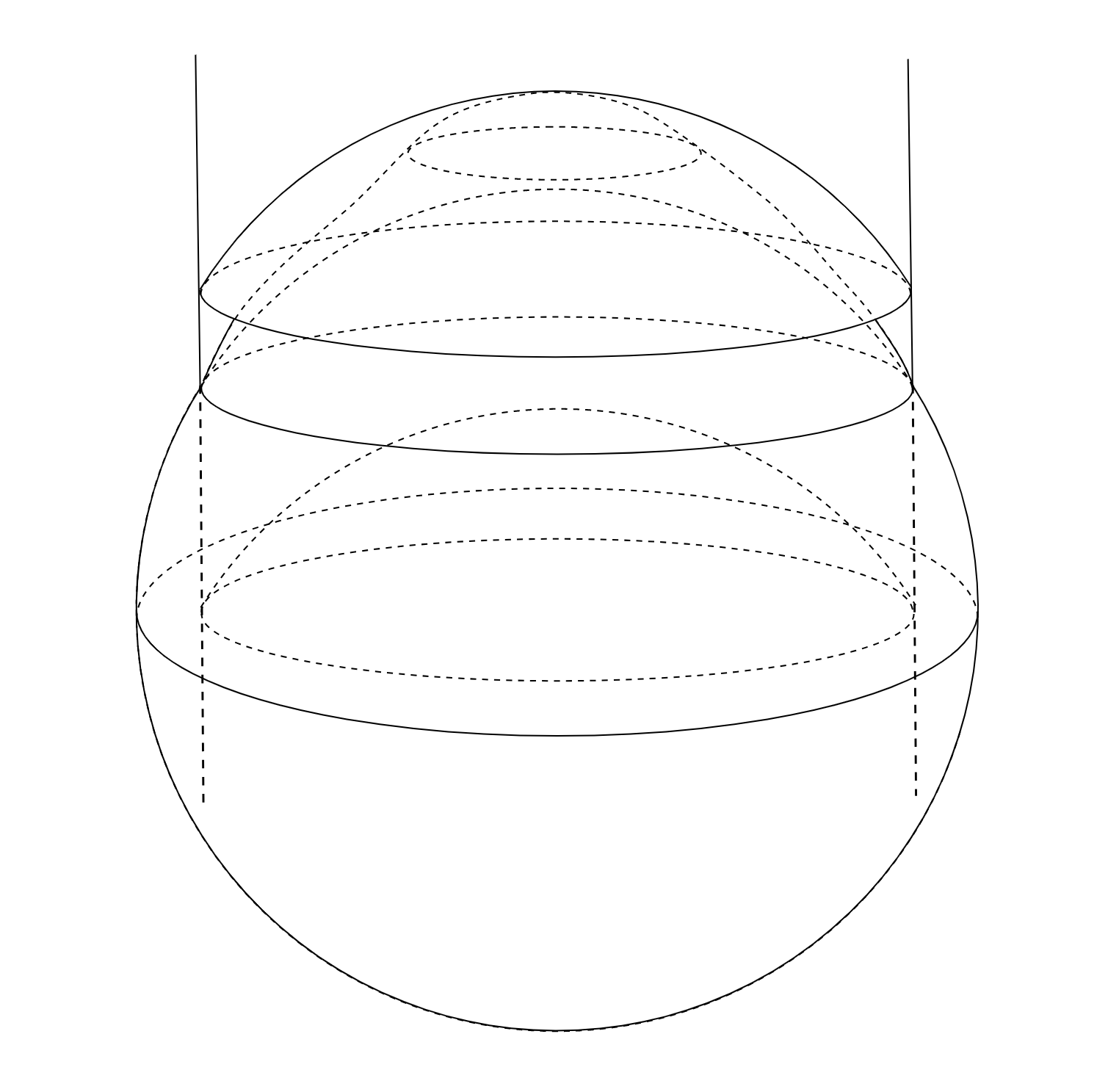}
   \label{fig1}
 \end{figure}

Consider the domain enclosed by $S^+$ and $C$, the round cylinder of radius 1:
\[
\Omega=\{(x_0,x_1,\cdots,x_n)\in \mathbb{R}^{n+1} |\Sigma_{i=1}^n (x_i)^2\le 1,  \sqrt{\Sigma_{i=1}^n (x_i)^2}< x_0 <+\infty~.\}
\]
This trap is sliced by a parallel family $\{F_t\}$ of hemispheres
which are just translations of $S^+$ upwards by $(t,0,\cdots,0)$ with $t>0$
and the same $H_r(F_t)=1$.
Compared to the statement of Theorem~\ref{thm-trap-slice} (the trap-slice lemma)
and Corollary~\ref{cor-trap-slice}, here we need only to take $B_0=S^+$,
and $B_1$ be the upper half cylinder. Then the conclusion follows immediately.
\end{proof}

We point out a general rigidity result for any CMC-graph in the Euclidean space. The proof is the same as above.

\begin{theorem}\label{thm-rigid-3}[The H-rigidity of a CMC graph]
For any compact hypersurface $B_0\subset \mathbb{R}^{n+1}$ which is a graph with constant $r$-mean curvature $H_r=\alpha$ and $\partial B_0=A$, it has no non-trivial smooth perturbation $\Sigma_0$ with the same boundary up to $C^2$ and $H_r(\Sigma_0)\le \alpha$ (or $H_r(\Sigma_0)\ge \alpha$).
\end{theorem}

\begin{remark}
It is possible to deform a minor arc of a circle or a small spherical cap outwards, to obtain non-trivial dumbbell-shaped perturbation $\Sigma_0$ with $H_r\le 1$, which is not a graph over the coordinate plane. See Section~6 for a description and the forth-coming paper \cite{MaWang}. But any of such perturbations can not be $C^0$-close to the original small cap according to the following version of rigidity theorem.
\end{remark}

\begin{theorem}\label{thm-rigid-4}
Let $\Sigma_0$ be a smooth perturbation of $S^+_{\theta}$ along the boundary up to $C^2$. Let $\Omega$ be the dumbbell region enclosed by $S^-_\theta$ and its mirror image $S^*_{\theta}=f(S^-_{\theta})$ (reflecting with respect to the hyperplane containing $\partial S^+_{\theta}$).
If $\Sigma_0$ is properly embedded and contained in $\Omega$, with mean curvature $H_r(\Sigma)\le 1$, then $\Sigma_0 = S^+_{\theta}$.
\end{theorem}
 \begin{figure}[!h]
   \setlength{\belowcaptionskip}{0.2cm}
   \centering
   \includegraphics[width=0.5\textwidth]{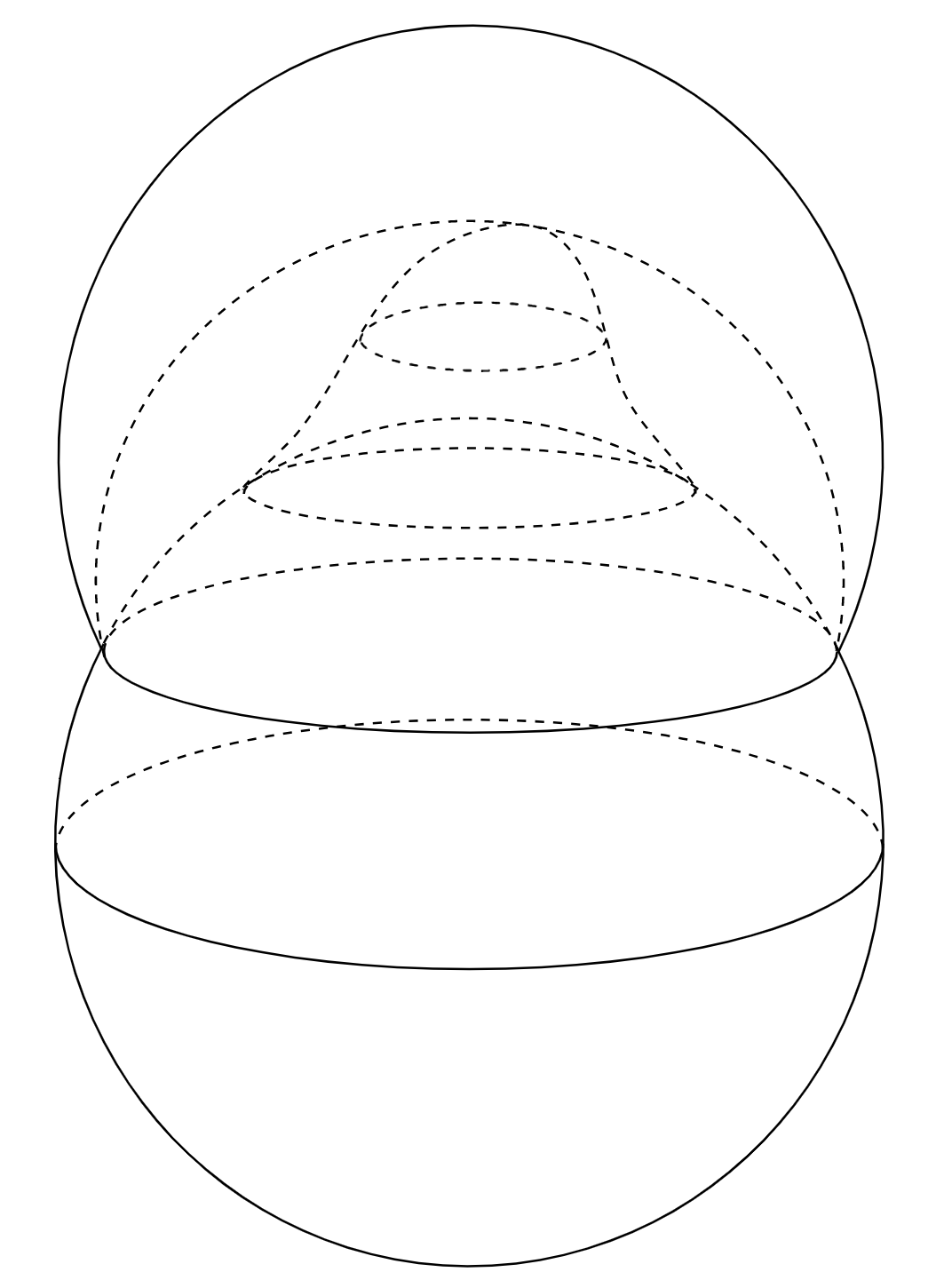}
   \label{fig1}
 \end{figure}

\begin{remark}
Notice that in the statement above, the perturbation $\Sigma_0$ of the small cap is not necessary a graph; moreover, its topological type is allowed to be different from a disk (a ball). On the other hand, the \textbf{properly embeddedness} condition is still necessary to guarantee that $\Sigma=\Sigma_0\cup S^-_{\theta}$ enclose a domain with well-defined inward normal vector field.
\end{remark}

\begin{proof}[Proof to Theorem~\ref{thm-rigid-4}]
The proof is similar to that of Theorem~\ref{thm-rigid-2} with only minor changes.
We need only to notice that the dumbbell-shaped trap $\Omega$ is still foliated by
a family of spherical caps with smaller radii, greater $H^r>1$, and the same boundary $\partial S^+_{\theta}$. In particular, $S^-_{\theta}, S^+_{\theta}$ and $S^*_{\theta}=f(S^-_{\theta})$ are all members in this family. By Corollary~\ref{cor-trap-slice}, the conclusion is clear.
\end{proof}

As a special case of these rigidity theorems above, we summarize and state the mean curvature rigidity of the hemisphere as below.
\begin{corollary}\label{cor-rigid-hemisphere}[The mean curvature rigidity of a hemisphere]
Let $\Sigma_0$ be a smooth perturbation in $\mathbb{R}^{n+1}$ of the hemisphere $S^+$ and the boundary is fixed up to $C^2$. We assume that $\Sigma=\Sigma_0\cup S^-$ is still a \textbf{convex} closed hypersurface (an ovaloid). \\
\indent 1) if the r-mean curvature of $\Sigma$ satisfies $H_r(\Sigma)\ge 1$, then $\Sigma_0 = S_+$;\\
\indent 2) or if the reversed inequality holds, $H_r(\Sigma)\le 1$, then $\Sigma_0 = S_+$.
\end{corollary}

\begin{remark}\label{remark-dumbbell}
When the perturbation $\Sigma$ is an ovaloid and the usual Gauss-Kronecker curvture $H_n\ge 1$,
we can prove $\Sigma=S^n$ without using the Tangency Principle (which means that one can avoid
using PDE or maximum principle).
Such a proof involves two ingredients: the total curvature integral and
the area(volume) comparison between two ovaloids (see discussions in Section~6).
The proof is easy and we omit it at here.
\end{remark}

\section{The non-rigidity results of a great spherical cap}

In contrast to the rigidity of small spherical caps and hemispheres, a great spherical cap has nontrivial perturbations $\Sigma_0$ with $H_r>1$ and $\tilde\Sigma_0$ with $H_r<1$. Here we will construct $\Sigma_0$ with $H_r>1$ first.

\begin{theorem}\label{thm-nonrigid-1}
When $\theta> \pi/2$, there exist a smooth hypersurface $\Sigma_0$ as a perturbation of the great spherical cap $S^+_{\theta}$ with the same boundary up to $C^2$, such that $\Sigma_0$ is a convex hypersurface of revolution and its principal curvatures $k_i> 1, \forall i=1,2,\cdots, n$ away from $\partial S^+_{\theta}$. In particular, this nontrivial perturbation has $H_r(\Sigma_0)> 1$ everywhere (away from the boundary).
\end{theorem}

\noindent
\textbf{Analysis.}
Let us consider the 1-dim case first. The higher dimensional case will follow easily since the principal curvatures for such a hypersurface of revolution are also determined by its profile curve on the $(z,y)$ plane (rotating around the $z$-axis).

Let $S=\overset{\frown}{EP_0F}$ be the original semi-circle (as the profile curve) with center $O$ and angle $\angle EOP_0=\pi-\theta<\frac{\pi}{2}$. Taking $OF$ to be the positive $z$-axis (the future positive $x_0$- axis). We perturb $S=\overset{\frown}{EP_0F}$ to curve $\Gamma=\overset{\frown}{EP_0P_4}$ so that the following conditions are satisfied: \\

\noindent
\textbf{Ansatz for $\Gamma$:}
\begin{enumerate}
  \item $\Gamma$ is the graph of a concave height function $y=y(z)$ with peak at $P_2$ (maximum of $y$). $\Gamma$ intersect with $OF$ at $P_4$ orthogonally.
  \item The union $\Gamma\cup -\Gamma$ is a convex $C^2$ closed curve on the $(z,y)$ plane, where $-\Gamma$ denotes the reflection of $\Gamma$ across the $z$-axis.
  \item The curvature $\kappa$ of $\Gamma$ is a unimodal function with maximum at $P_2$ and minimum $\kappa=1$ on $\overset{\frown}{EP_0}$. In particular, $\kappa(P_4)>1$ is a local minimum on $\Gamma\cup -\Gamma$.
\end{enumerate}

To find such a $\Gamma$, the real challenge is to control the monotonicity and the minimum of $\kappa$. This is easy to do by examining the loci of the corresponding center of curvature $\gamma$. We summarize the classical properties of the correspondence between $p\in \gamma$ and $P\in \Gamma$ as below \cite{do-Carmo}.

\begin{proposition}\label{prop-evolute}[evolute-involute correspondence]
$\gamma$ is known as the \emph{evolute} of $\Gamma$. They are related with each other and having the following properties:
\begin{itemize}
  \item The line segment $pP$ is tangent to $\gamma$ at $p$ and normal to $\Gamma$ at $P$; its length $|\overline{Pp}|=\frac{1}{\kappa}=R$ is the curvature radius of $\Gamma$ at $P$.
  \item $\Gamma$ is called the \emph{involute} of $\gamma$ with the following property: for two points $P_i, P_j\in \Gamma$ and the corresponding points $p_i, p_j\in \gamma$, there is
      \[
      L(\overline{P_j p_j})=L(\overline{P_i p_i})+ L(\overset{\frown}{p_ip_j}), ~~~\text{(with the orientation kept in mind).}
      \]
      for the lengths of the line segments and the arc-length on $\gamma$. In other words, with respect to an arc-length parameter $t$ of $\gamma=p(t)$ and some suitable constant $a$, $P\in\Gamma$ is given by
    \begin{equation}\label{eq-involute}
      P(t)=p(t)+p'(t)(a-t)
    \end{equation}
  \item At a vertex $P$ of $\Gamma$ where the curvature $\kappa$ attains maximum (or minimum), the evolute $\gamma$ has a cusp $p$.
\end{itemize}
\end{proposition}

 \begin{figure}[!h]
   \setlength{\belowcaptionskip}{0.2cm}
   \centering
   \includegraphics[width=0.7\textwidth]{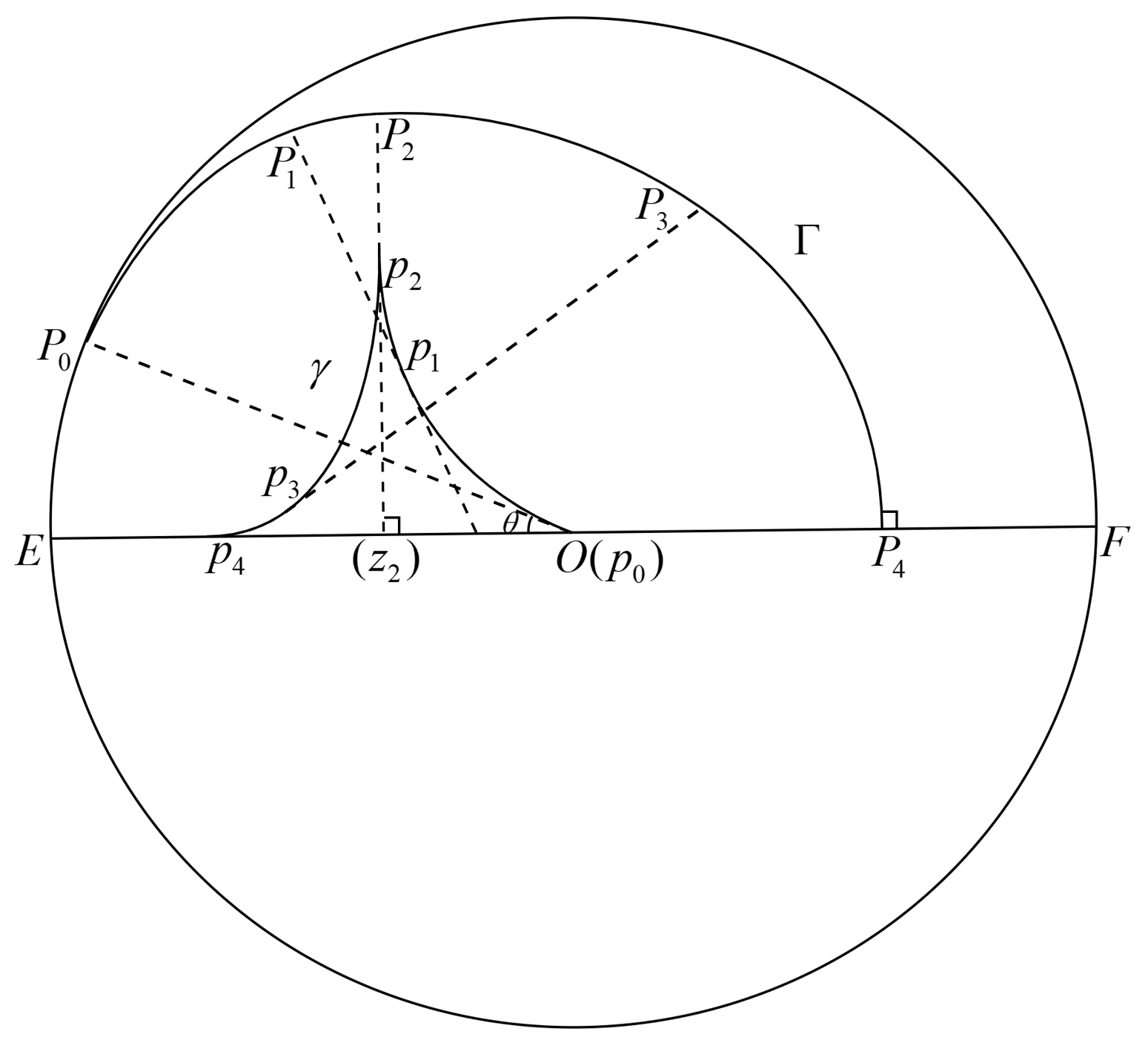}
   \label{fig1}
 \end{figure}

Under this correspondence, the conditions on $\Gamma$ are converted to the following

 \noindent
\textbf{Ansatz for the evolute $\gamma$:}(indicated in the figure)
\begin{enumerate}
  \item $\gamma=\overset{\frown}{p_4p_2p_0}$ is a piecewise $C^2$ curve with $p_0=O$ at the origin, $p_4=(z_4,0)\in EO$. It is also the graph of a non-negative unimodal function $h=h(z)$ with the peak at $p_2=(z_2,h_2)$.
  \item The union $\gamma\cup -\gamma$ on the $(z,y)$ plane looks like an astroid, with exactly three cusps at $p_4, p_2,-p_2$. At $p_0=O$ the tangent line of $\gamma$ is the same as $P_0 O$, thus $h'(0)=\tan\theta$.
  \item $L(\overset{\frown}{p_4p_2})<L(\overset{\frown}{p_2p_0})<1$, ~~~~
by Proposition~\ref{prop-evolute} and the condition $R_4<1$,
      \[
      R_4=L(\overline{p_4P_4})=L(\overset{\frown}{p_4p_2})+L(\overline{p_2P_2}),~~
      1=R_0=L(\overline{p_0P_0})=L(\overset{\frown}{p_0p_2})+L(\overline{p_2P_2}).
      \]

\end{enumerate}
Now the existence of such an evolute is clear. The rest is to check the details on the curvature functions of the profile curve $\Gamma$ and the hypersurface of revolution $\Sigma$.

\begin{proof}[Proof to Theorem~\ref{thm-nonrigid-0}].

\textbf{Step 1: Construction of the evolute $\gamma$.}

It is easy to see that there exists a piecewise smooth curve $\gamma: (z, h(z))$ as the graph of a function $h$ satisfying all of the requirement below:

(1) $h(z)$ is a convex function defined separately on the interval $[z_*,z_2]$ and $[z_2,0]$ where $-1<2z_2<z_*<z_2<0$.

(2) On the interval $[z_2,0]$, $h(z)$ decreases monotonically from a maximal value $h_2\in (0,1)$ to $0$, with $h'(z_2)=-\infty, h'(0)=\tan\theta$.

(3) On the interval $[z_*,z_2]$, $h(z)$ increases monotonically from $0$ to the same maximal value $h_2$, with $h'(z_*)=0, h'(z_2)=+\infty$.

(4) The end points $p_4\triangleq (z_*,0)$ and $p_2\triangleq (z_2,h_2)$ are both cusps of $\gamma$, i.e. we have asymptotic expressions separately as
\[
h(z)\simeq (z-z_*)^{3/2}+O(|z-z_*|^2);~~~~h(z)\simeq h_2-(z-z_2)^{2/3}+O(|z-z_2|).
\]

(5) The arc-length of $\gamma$ on the interval $[z_*,z_2]$ is smaller than that on $[z_2,0]$, and the latter is small than 1, i.e., $L(\overset{\frown}{p_4p_2})<L(\overset{\frown}{p_2p_0})<1$.\\

\textbf{Step 2: Construction of the involute $\Gamma$.} Now we construct the involute $\Gamma$ of the tent-like curve $\gamma$.

Consider a moving point $p_1(t)\in \overset{\frown}{p_2p_0}\subset \gamma$  with arc-length parameter $t$ of $\gamma$, such that $t=0$ at $p_0=O$ and $t=L(\overset{\frown}{p_2p_0})$ at $p_2$. We write
\begin{equation}\label{eq-involute-1}
P_1(t)\triangleq p_1(t)+p'_1(t)(1-t).
\end{equation}
$P_1$ traces a convex curve $\Gamma$ which starts at $P_0=(\cos\theta,\sin\theta)$ and ends at
$P_2=(z_2,y_2)$ with $y_2=h_2+1-L(\overset{\frown}{p_2p_0}))<1$. It is connected to the minor circular arc $\overset{\frown}{EP_0}$ at $P_0$ smoothly ($C^2$).
Along $\Gamma$ from $P_0$ to $P_2$, the curvature radius $1-t$ decreases from 1 to $R_2\triangleq 1- L(\overset{\frown}{p_2p_0})>0$; hence $\kappa$ increases monotonically to the maximum $\kappa(P_2)=1/R_2>1$.

Similarly, let $p_3(s)$ be the arc-length parametrization of the left half of $\gamma$, such that $s=0$ at $p_2$ and $s=L(\overset{\frown}{p_4p_2})$ at $p_0$. When $p_3$ moves along $\overset{\frown}{p_4p_2}\subset \gamma$, the right part of the involute $\Gamma$ is traced by
\begin{equation}\label{eq-involute-2}
P_3(t)=p(t)-p'(t)(t+R_2),~~~R_2=L(\overline{p_2P_2}).
\end{equation}
It satisfies the \textbf{Ansatz for $\Gamma$}. In particular, along $\Gamma$ from $P_2$ to $P_4$, the curvature radius $t+R_2$ increases from $R_2$ to $R_4=R_2+L(\overset{\frown}{p_4p_2})<1$; hence $\kappa(\Gamma)$ decreases monotonically to a local minimum $1/R_2 >1$.

The left details are easy to verify and we conclude $\kappa(\Gamma)>1$ on the perturbed part from $P_0$ to $P_4$.\\

\textbf{Step 3: The hypersurface $\Sigma$ of revolution.}

Let $S^n$ be perturbed to $\Sigma=\Sigma_0\cup S^+_{\theta} (\theta>\pi/2)$ which is still a hypersurface of revolution with profile curve $\Gamma: (z,y(z))$. Then $\Sigma\subset \mathbb{R}^{n+1}$ can be parameterized as
\[
(z,v_1,\cdots,v_{n-1})~\mapsto (z,y(z)\cdot\Theta(v_1,\cdots,v_{n-1}))
\]
where $\Theta(v_1,\cdots,v_{n-1})$ is a parametrization of the $n-1$ dimensional round sphere.
It is well-known that in this case, one principal curvature $\kappa_0$ is the same as $\Gamma$, and the other principal curvature $\kappa_1=\cdots=\kappa_{n-1}$ has multiplicity $n-1$ whose reciprocal $R_1=1/\kappa_1$ is the distance on the normal line between $P_1\in \Gamma$ and $q_1$ (the intersection point of the line $P_1p_1$ with $z$-axis).

For the right branch of $\gamma$, according to the figure and the convexity, it is clear that the line segment $\overline{P_1 q_1}$ is always shorter than the path $\overline{P_1p_1}\cup \overset{\frown}{p_1p_0}$ whose length equals $L(\overline{P_1p_1})+L(\overset{\frown}{p_1p_0})=1=R_0$ (by the involute-evolute property, see Proposition~\ref{prop-evolute}). Thus $R_1<1$ is always true and $\kappa_1>1$ (except at $P_0$ and $p_0=O$). On the left branch $\overset{\frown}{p_2p_4}\subset \gamma$ and the corresponding $\overset{\frown}{P_2P_4}\subset \gamma$, this is proved similarly.

In summary, we know $\kappa_1=\cdots=\kappa_{n-1}>1$ and $\kappa_0=\kappa(\Gamma)>1$ away from the fixed small spherical cap. This completes the proof.
\end{proof}

We can adopt the same idea and construct the perturbation with smaller principal curvatures.

\begin{theorem}\label{thm-nonrigid-2}
When $\theta> \pi/2$, there exist a smooth hypersurface $\tilde\Sigma_0$ as a perturbation of the great spherical cap $S^+_{\theta}$, such that $\tilde\Sigma_0$ is a convex hypersurface of revolution and its principal curvatures $0<k_i< 1, \forall i=1,2,\cdots, n$ away from $\partial S^+_{\theta}$. In particular, this nontrivial perturbation has $H_r(\tilde\Sigma_0)<1$ everywhere.
\end{theorem}

The proof is almost the same as the previous one. So we omit the detail and only demonstrate the figure of the possible construction.

\begin{figure}[!h]
   \setlength{\belowcaptionskip}{0.2cm}
   \centering
   \includegraphics[width=0.7\textwidth]{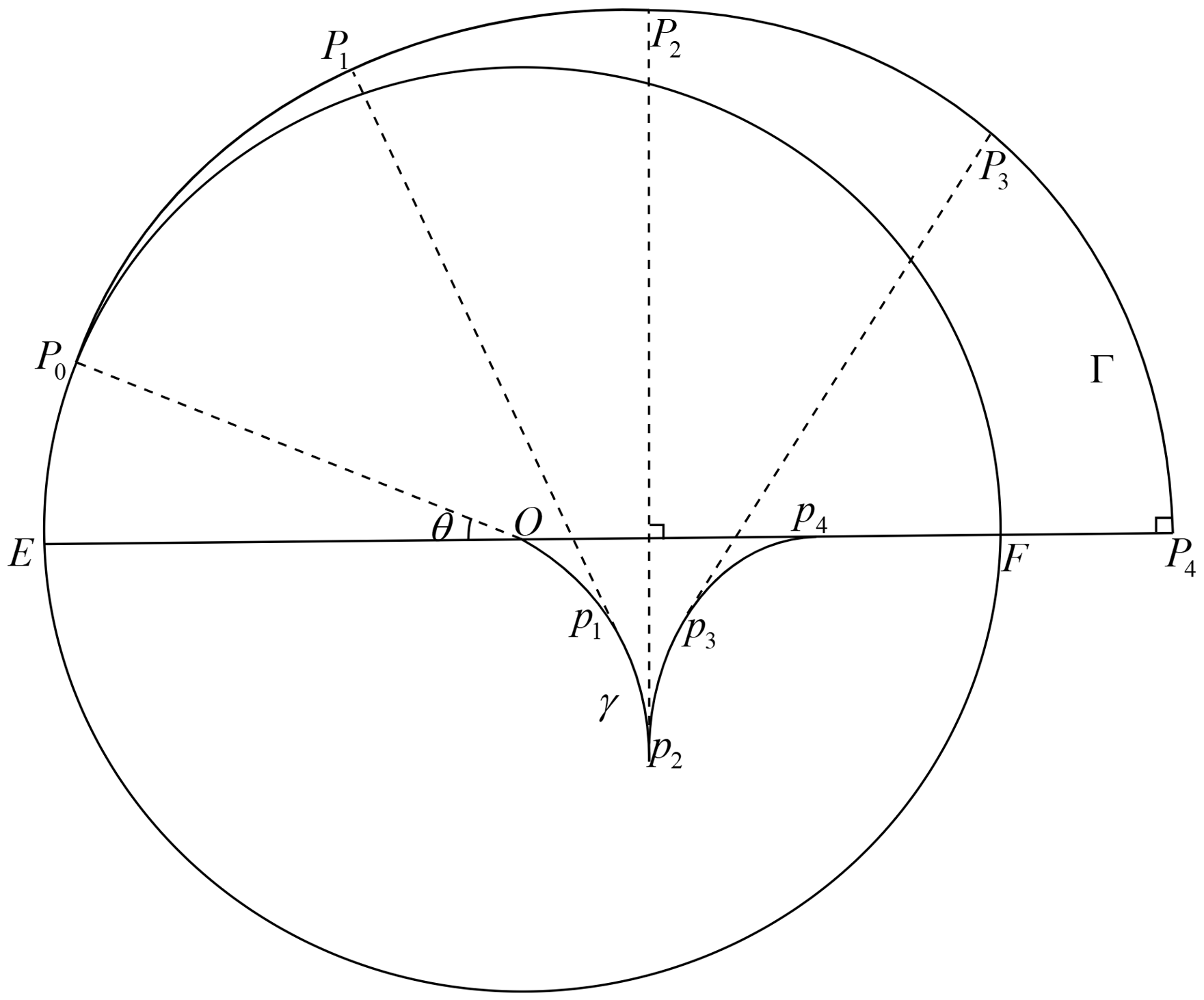}
   \label{fig1}
\end{figure}

\section{Final remarks}
We discuss various possible extensions and raise several open questions. Some partial results will also be announced at here.

\subsection{Relative rigidity of ovaloids}

By ovaloid we mean a closed convex hypersurface in $\mathbb{R}^{n+1}$. For two ovaloids whose curvatures are not necessarily constants, we find a relative version (in the sense of Blaschke) of the \emph{curvature rigidity} as below.

\begin{theorem}\label{thm-rigid-5}[Relative rigidity of ovaloids]
Let $\Sigma$ and $\tilde\Sigma$ be two closed convex hypersurfaces in $\mathbb{R}^{n+1}$ with inward normal vectors. View the Gauss-Kronecker curvature $K=K(\nu)$ and $\tilde{K}=\tilde{K}(\nu)$ as functions depending on the unit normal vector $\nu\in S^n$ (the Gauss image). Suppose $K(\nu)\ge \tilde{K}(\nu)$ for $\nu$ contained in a hemisphere, and $K(\nu)=\tilde{K}(\nu)$ otherwise. Then $\Sigma$ is congruent to $\tilde\Sigma$ up to a translation and $K(\nu)=\tilde{K}(\nu)$ everywhere.
\end{theorem}

At the beginning, we just wanted to generalize the mean curvature rigidity theorem to the case that the curvature function is not a constant. In the previous arguments, it seems that one only need to compare the curvatures of $\Sigma$ and $\tilde\Sigma$ at points with the same normal direction.
This motivated us to conjecture that Theorem~\ref{thm-rigid-5} might be true.

Indeed, Theorem~\ref{thm-rigid-5} follows directly from the well-known integral identity below (related with the classical Minkowski Problem):
\[
\int_{S^n} \frac{x_i}{K(\nu)} d\nu =\int_{\Sigma} \nu \cdot E_i ~d Vol(\Sigma)=0,
\]
where $x_i$ are the coordinate functions and $E_i$ is the standard $i-$th coordinate vector of $S^n$.

On the other hand, given any ovaloid $\Sigma$ with positive Gauss-Kronecker curvature $K(\nu)>0$,
it is always possible to find a perturbation of it, denoted as $\tilde\Sigma$ with Gauss-Kronecker curvature $\tilde{K}(\nu)>0$,
such that $K(\nu)=\tilde{K}(\nu)$ when $\nu$ is restricted in a small spherical cap of $S^n$,
and $K(\nu)<\tilde{K}(\nu)$ (or $K(\nu)>\tilde{K}(\nu)$) on the complementary great spherical cap.
We need only to construct $\tilde{K}(\nu)$ on $S^n$ such that
$\int_{S^n} \frac{x_i}{K(\nu)} d\nu =0$ and $K(\nu)=\tilde{K}(\nu)$ on the given small cap, which is not difficult.
Then by the positive answer to the Minkowski problem (by the work of Nirenberg, Pogorelov, and Cheng-Yau), there exists such an ovaloid $\tilde\Sigma$ with $\tilde{K}$ as the desired Gauss-Kronecker function. So the relative rigidity is not true in this case.

\begin{remark}
It seems that Theorem~\ref{thm-rigid-5} can be proved by the \emph{moving plane} method similarly,
combining with the basic facts on the curvature integral of the Gauss-Kronecker curvature $K$
(which is a constant when the perturbation is a regular homotopy and the boundary is fixed up to $C^2$)
and the well-known \emph{Area Comparison} between two nested ovaloids (the outside one has larger area).
But we should point out a hidden flaw in this method.
One tends to believe that $\Sigma=\tilde\Sigma$ on the subset with the same Gauss-Kronecker curvature $K(\nu)=\tilde{K}(\nu)$ when their supporting planes are parallel.
Yet this is not a condition in the statement of Theorem~\ref{thm-rigid-5},
and it might be wrong if we restrict only to smooth convex hypersurfaces with boundaries.

For example, considering two elliptic paraboloids $r(u,v)=(au,bv,\frac{1}{2}(au^2+bv^2))$ and
$r_*(u,v)=(a_*u,b_*v,\frac{1}{2}(a_*u^2+b_*v^2))$ with parameters $(u,v)$ and positive constants $a,b,a_*,b_*$ such that $ab=a_* b_*$. Calculation shows that they have the same normal direction and equal Gauss curvature on points with the same parameter $(u,v)$; yet these two surfaces are not isometric.
\end{remark}


We can also consider the relative rigidity problem for the mean curvature $H$.

\begin{problem}\label{problem-relative-H}
Let $\Sigma$ and $\tilde\Sigma$ be two closed convex hypersurfaces in $\mathbb{R}^{n+1}$. View the mean curvature $H=H(\nu)$ and $\tilde{H}=\tilde{H}(\nu)$ as positive functions depending on the unit normal vector $\nu\in S^n$ (the Gauss image). Suppose $H(\nu)\ge \tilde{H}(\nu)$ for $\nu$ contained in a subset of the hemisphere, and $H(\nu)=\tilde{H}(\nu)$ on the other hemisphere. Must $\Sigma$ be congruent to $\tilde\Sigma$ up to a translation and $H(\nu)=\tilde{H}(\nu)$ everywhere?
\end{problem}

\subsection{Nontrivial deformations of small spherical caps}

If we allow $\Sigma$ to have intersection, then it is not difficult to construct counter-examples with self-intersection and with larger mean curvature $H>1$ on the perturbed small spherical cap (at least in the 1-dimensional case it is easy to imagine a closed curve with winding number=2; it is not regularly homotopic to the original circle).

Moreover, we can construct a family of connected and properly embedded hypersurfaces of revolution $\Sigma$ as perturbations of the hemisphere, which has $H<1$ on an open subset and it is not a graph over the coordinate plane $x_0=0$. The shape looks like a \emph{dumb bell}. The detail of the construction will appear in the forthcoming paper \cite{MaWang}; but intuitively it is easy to imagine and a picture is given below.

 \begin{figure}[!h]
   \setlength{\belowcaptionskip}{0.2cm}
   \centering
   \includegraphics[width=0.7\textwidth]{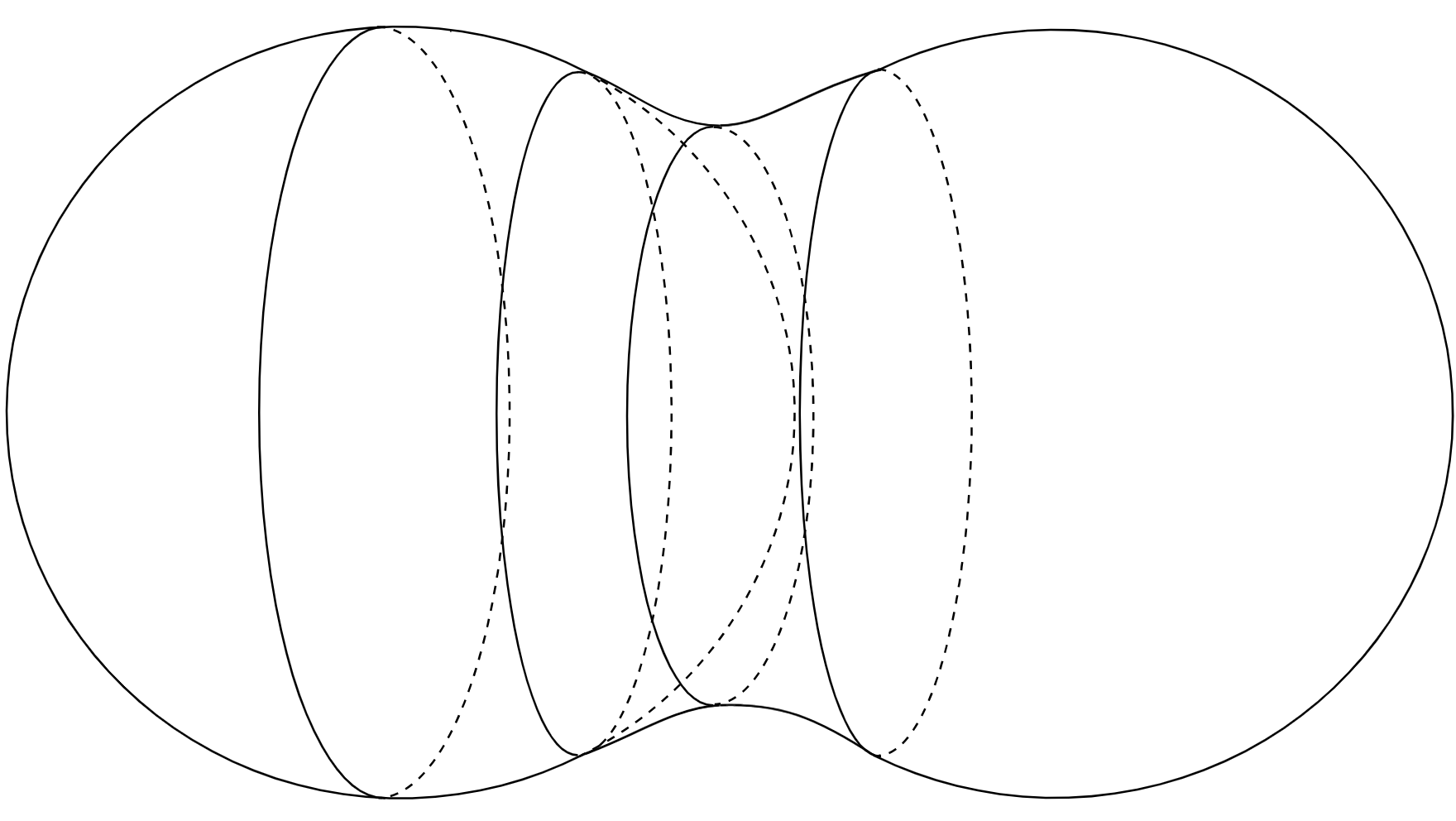}
   \label{fig1}
 \end{figure}

This shows that the graph condition in the statement of Theorem~\ref{thm-rigid-2}, as well as the dumbbell-domain restriction $\Sigma\subset \Omega$ in the statement of Theorem~\ref{thm-rigid-4}, are both necessary, and these two theorems are best possible.

\subsection{Perturbations on other sphere region}
Can we consider more general perturbations on other subsets of the sphere? The easiest case next to a cap is a double connected domain, like an annulus around the equator.

In this respect, there was an interesting intrinsic result in \cite{HangFB-06} (Theorem 2.8 in it).
Here $B(S, a)$ and $B(N, a)$ are geodesic balls with radius $b$ and centered at the north pole $N$ and south pole $S$, respectively.

\begin{theorem}[F.B. Hang and X.D. Wang 2006 \cite{HangFB-06}]
For any $a \in (0,\pi/2)$, there exists a smooth metric $g$ on $S^n$
such that $g = g_{S^n}$ on $B(S, a)\cup B(N, a)$ and the sectional curvature of $g$ is
at least $1$ and larger than $1$ somewhere.
\end{theorem}

It is natural to ask:

\begin{problem}\label{problem-equator}
Can one perturb $S^n\subset \mathbb{R}^{n+1}$ in a given neighborhood $U$ of the equator $S^{n-1}$ such that:\\
 \indent (1) the perturbation $\Sigma$ still agree with $S^n$ outside $U$, and \\
 \indent (2) $\Sigma$ has sectional curvature larger than $1$ inside $U$?
\end{problem}

For this problem, surprisingly we have both rigidity and non-rigidity result. Consider $a \in (0,\pi/2), b=\pi/2-a$ and $B(S, a)\cup B(N, a)\subset S^n\subset \mathbb{R}^{n+1}$ with $n\ge 2$. The supplementary subset $B_0$ (a closed annulus with width $2b=\pi-2a$) is a neighborhood of the equator $S^{n-1}$ with radius $b$.

Our result in a forthcoming paper \cite{MaWang} says that there exists a critical value $a_0>0$ (dependent on $n$) such that:
\begin{itemize}
  \item When $a\in (0,a_0)$, there is no smooth perturbation of $S^n$ which fixes $B(S, a)\cup B(N, a)$ and has sectional curvature $K>1$ somewhere;
  \item When $a\in (a_0,\pi/2)$, there exist surfaces of revolution as smooth perturbations of $S^n$ which fix $B(S, a)\cup B(N, a)$ and has sectional curvature $K>1$ on the annulus $B_0$ with width $2b$.
\end{itemize}

\subsection{Rigidity or Non-rigidity? Open questions}

We can also consider perturbations of other kind of minimal/CMC surfaces with compact support in the space forms. In most of these cases, it is interesting to determine the borderline case between rigidity and non-rigidity.

\begin{problem}\label{problem-cmc}
Let $M$ be the round cylinder or the catenoid. Consider a perturbation of $M$, denoted as $\Sigma$, has larger (or smaller) mean curvature $H$ on the perturbed domain which has compact support $B_0$.
On what kind of domain such nontrivial $\Sigma$ is allowed?
\end{problem}

We have also obtained some partial results in this direction. In particular, for the cylinder case the discussion is closely related with the Delaunay surfaces (CMC surfaces of revolution).

\begin{remark}
It should be not difficult to generalize the mean curvature rigidity and non-rigidity results to a geodesic sphere $S^n=\partial B^{n+1}(r)$ in other space forms, where $B^{n+1}(r)$ denotes a geodesic ball of radius $r$. For the hyperbolic space $\mathbb{H}^{n+1}(-1)$ of constant sectional curvature $-1$ the conclusions should be similar. On the other hand, in the standard sphere $\mathbb{S}^{n+1}(1)$ of constant sectional curvature $1$, if $r<\pi/2$, we conjecture that the rigidity theorem is still true if the perturbation $\Sigma$ is restricted inside a hemisphere $ B^{n+1}(1)\subset \mathbb{S}^{n+1}$. The special case when $r=\pi/2$ (the equator) is subtler and we leave it to interested readers.
\end{remark}

Finally, related with Gromov's original theorem at the beginning, we ask a question, which seems still open. We conjecture that the answer is affirmative.

\begin{problem}
Let $\Sigma$ is a smooth perturbation of the hyperplane $M=\mathbb{R}^n \subset \mathbb{R}^{n+1}$, which is contained between two parallel hyperplanes $M_0,M_1// M$. Must $\Sigma$ be also a hyperplane parallel to $M$?
\end{problem}

\section{Appendix: The discrete version for polygons}
In this independent appendix, we treat a discrete version of the rigidity problem in the 1-dim case. Instead of a circle, we will consider convex polygon inscribed in a circle and its perturbations. For such objects there is also a notion of \emph{discrete curvature} at each vertex, called \emph{the Menger curvature} \cite{Jecko}. We will show that one can not increase this discrete curvature on a convex polygon inscribed in a minor arc or a semi-circle.

\begin{definition}\cite{Jecko}
For a convex polygon $A_1A_2 \cdots A_n(n \ge 3), 1 \le i \le n$ for $i\in \mathbb{Z}/n\mathbb{Z}$, the radius of the circumscribed circle of $\bigtriangleup A_{i-1} A_{i} A_{i+1}$ is denoted as $R_i\triangleq R_{\bigtriangleup A_{i-1} A_{i} A_{i+1}}$. The \emph{Menger curvature} at vertex $A_i$ is the reciprocal $M(A_i)=1/R_i$.
\end{definition}

 \begin{figure}[!h]
   \setlength{\belowcaptionskip}{0.7cm}
   \centering
   \includegraphics[width=0.5\textwidth]{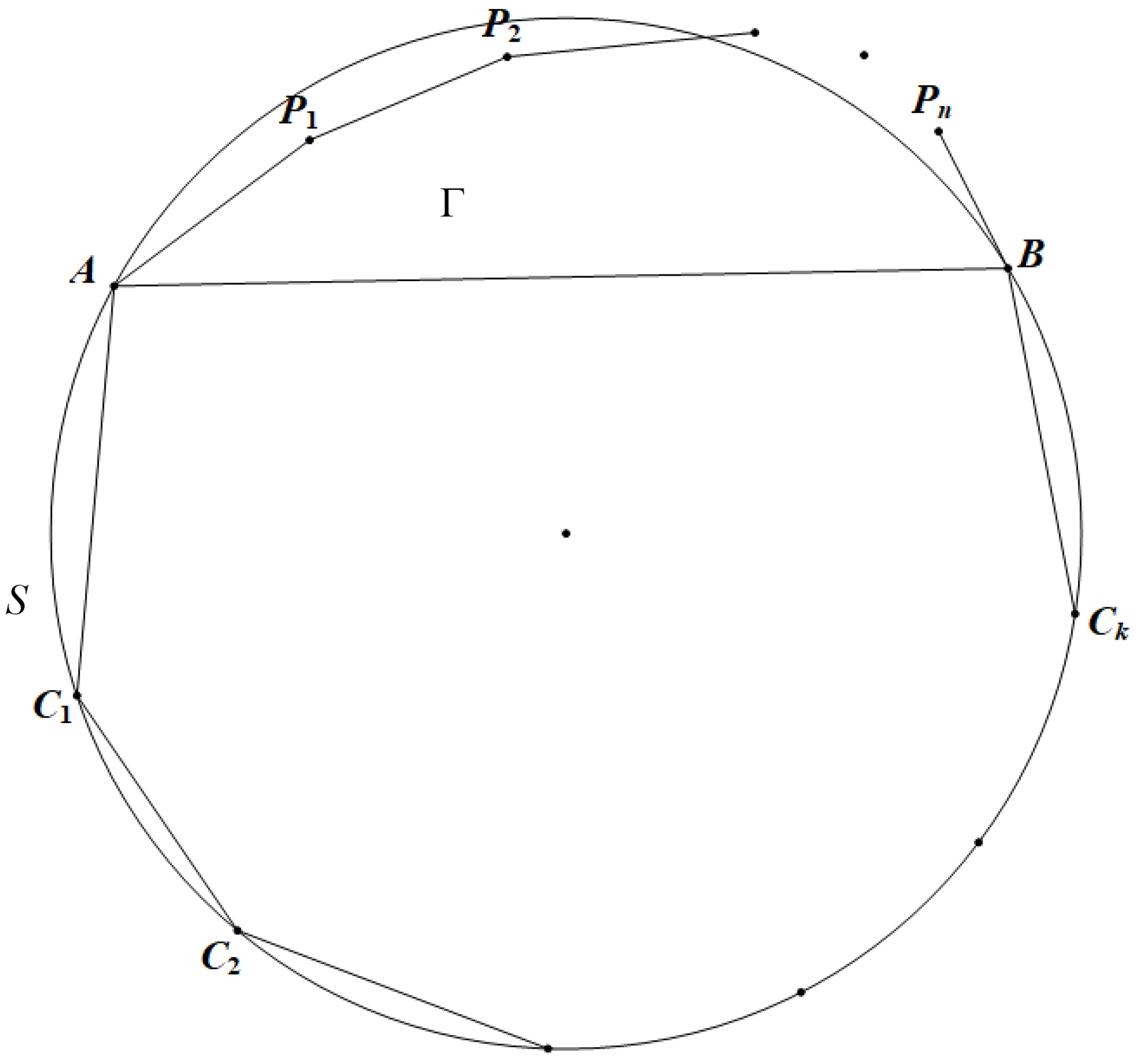}
   \caption{Perturbation of the polygon}
   \label{fig1}
 \end{figure}
As in Figure~1, $S_+=\overset{\frown}{AB}$ is a minor arc or a semicircle of the circle $S$ with radius $R$. This means that the central angle of $\overset{\frown}{AB}$ is not greater than $\pi$. Consider the convex polygon (as a perturbation of the \emph{discrete circle}) given by
\begin{align*}
\Gamma=AP_1P_2 \cdots P_n BC_k C_{k-1} \cdots C_1
\end{align*}
with $k,n \ge 1$, and $C_1,C_2 \cdots C_k \in S_-\triangleq S\setminus\overset{\frown}{AB}$ are located on the supplementary major arc of $S$ in order.

\begin{theorem}\label{thm-discrete}
In $\Gamma$, if the Menger curvatures of $A,P_1, P_2 \cdots P_n, B$ are all greater than or equal to $1/R$, then $M(A)=M(P_i)=M(B)=1/R$, and $\Gamma$ is inscribed in $S$.
\end{theorem}
\textbf{Analysis:}
$\overline{AB}$ divides the closed disk enclosed by $S$ into two arches. $D_-$ is the lower arch containing $\{C_i\}\subset S_-$, and $D_+$ is the upper arch with $\partial D_+=S_+\cup \overline{AB}$. One of $\overset{\frown}{AC_1}$ and $\overset{\frown}{BC_k}$ must be a minor arc; thus without loos of generality, we assume $\overset{\frown}{AC_1}$ is a minor arc.

To prove the theorem, it is clear that we only need to prove the following

\begin{lemma}
As in Figure~2, consider $\Gamma$ with $k=1$ and $C_1=C$. If $\overset{\frown}{AC}$ is a minor arc in $S$, and the Menger curvatures at $A,P_1, P_2 \cdots P_n$ are all greater than or equal to $1/R$, then the Menger curvatures are all equal to $1/R$, and $\Gamma$ is inscribed in $S$.
\end{lemma}

\begin{proof}[Proof of the lemma (and the theorem)]\par
First we prove that if $M(A) \ge 1/R$, then $P_1 \in D_+$. Otherwise, as in the figure, let $P_1 \notin D_+$ and denote the intersection point $Q_1\triangleq CP_1 \cap S_+=\overset{\frown}{AB}$.
Then $\angle AP_1C < \angle AQ_1C < \pi/2$ for the minor arc $\overset{\frown}{AC}$.
Thus $R_{\bigtriangleup AP_1C} > R_{\bigtriangleup AQ_1C} = R$, and $M(A) < 1/R$, which is a contradiction.

We prove the lemma by induction on $n$, the number of the vertices $\{P_i\}$.

\begin{figure}[htbp]
\centering
\subfigure[Figure 2.]{
\begin{minipage}[t]{0.5\linewidth}
\centering
\includegraphics[width=1.0\textwidth]{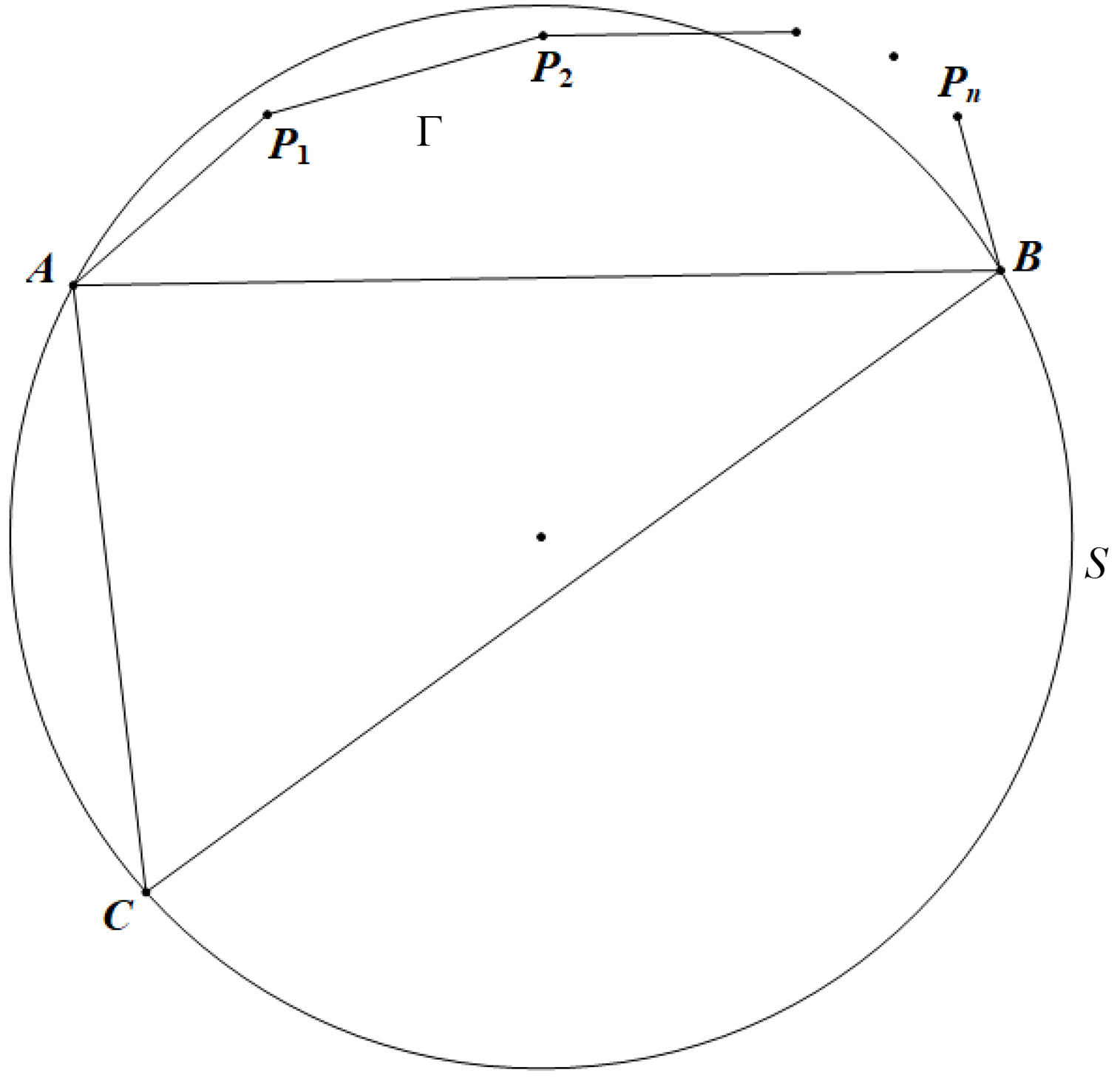}
\end{minipage}
}%
\subfigure[Figure 3.]{
\begin{minipage}[t]{0.5\linewidth}
\centering
\includegraphics[width=1.0\textwidth]{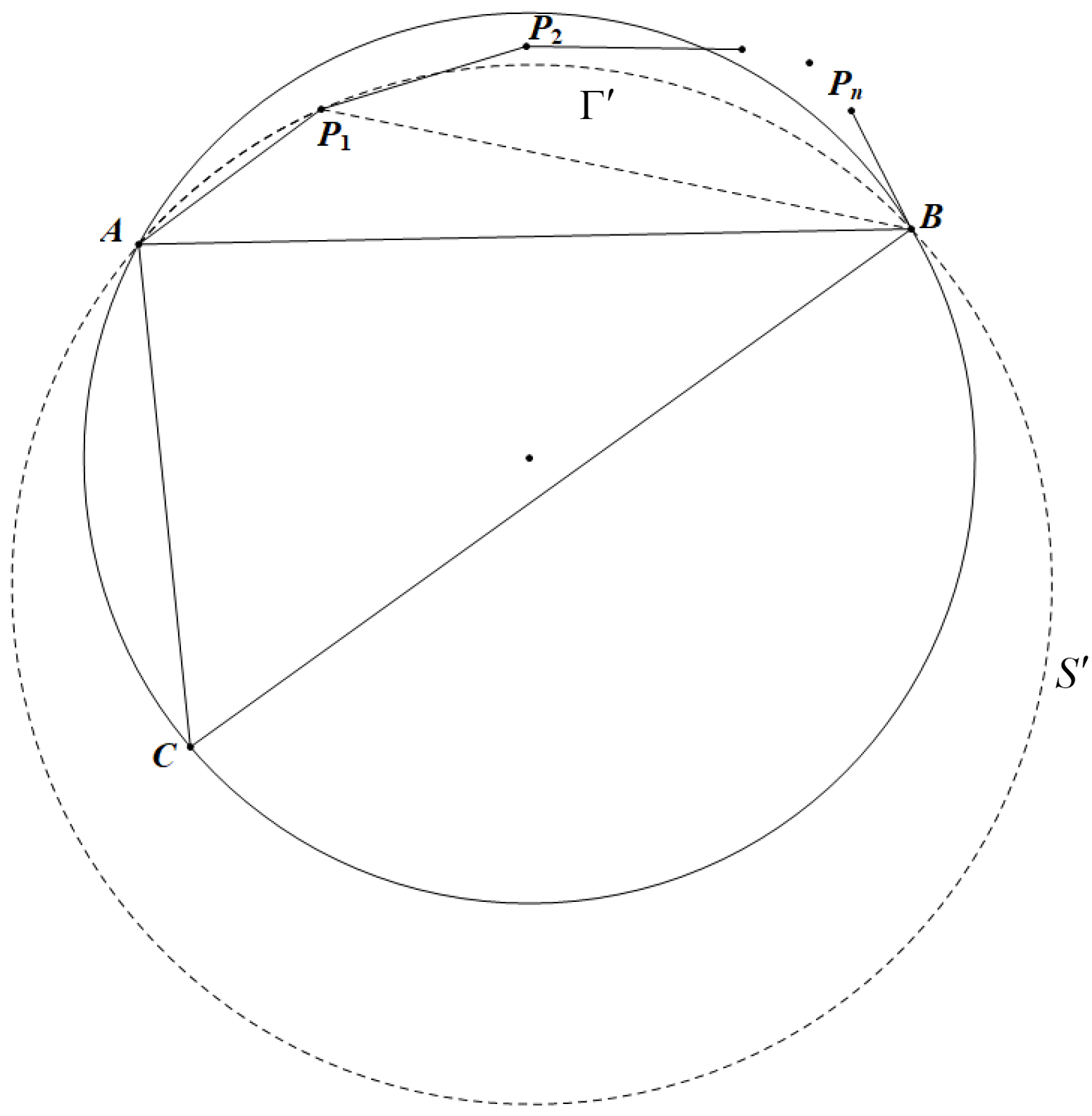}
\end{minipage}
}%
\end{figure}
When $n=1$, the theorem is obviously true, because if $P_1 \in \mathring{S_+}$, then $M(P_1) < 1/R$ for the minor arc $\overset{\frown}{AB}$, which violates the assumption $M(P_1)\ge 1/R$. So there must be $P_1 \in \overset{\frown}{AB}$.

Next, we assume that the lemma holds for $n=N$, and consider $n = N+1$. Define $S'$ as the circumscribed circle of $\bigtriangleup AP_1C$ with radius $R'$. It is clear that $R' \ge R$. Another thing is obvious that $\angle AP_1 B > \pi/2$, so $\overset{\frown}{AC}$ and $\overset{\frown}{AP_1}$ are all minor arcs in $S'$.

As in Figure~3, inside the circle $S'$ let us consider the chord $P_1 B$ and the convex polygon $\Gamma' =P_1 \cdots P_{N+1} B$. By assumption, the Menger curvature at $P_1, \cdots, P_{N+1}$ in $\Gamma'$ are not less than $1/R\ge 1/R'$.

The inductive hypothesis implies that the Menger curvature of $P_1 \cdots P_{N+1}$ in $\Gamma'$ are all equal to $1/R'$. Finally, by assumption they are not less than $1/R$, and $1/R \ge 1/R'$, we conclude that those Menger curvatures are all equal to $1/R$, and $R'=R$, which ends the proof when $n = N+1$. Thus we have proved the lemma and the theorem follows.
\end{proof}

\begin{remark}
The theorem requests that the half central angle $\theta=\frac{1}{2}\angle AOB$ is not greater than $\pi /2$. Actually, when $\theta>\pi / 2$, the theorem will not be true. Consider the unit circle  $S^1 \subset \mathbb{R}^2$, and $A,B \in S^1$, with $y_A = y_B <0$, $C(0,-1),D(0,1)$ in Figure~4. Then we move $D$ to $D'(0, y_0)$, with $0 <y_0 <1$ and $\angle AD'B< \pi / 2$. It is easy to vertify that $M(A),M(D'),M(B) >1$ in the polygon $ACBD'$. This agrees with the non-rigidity of the major circle and the great spherical cap as mentioned before.
\end{remark}

\begin{figure}[htbp]
  \centering
  \subfigure[Figure 4.]{
    \begin{minipage}[t]{0.4\linewidth}
    \centering
    \includegraphics[width=1.0\textwidth]{04.png}
    \end{minipage}
    }
  \subfigure[Figure 5.]{
    \begin{minipage}[t]{0.4\linewidth}
    \centering
    \includegraphics[width=0.7\textwidth]{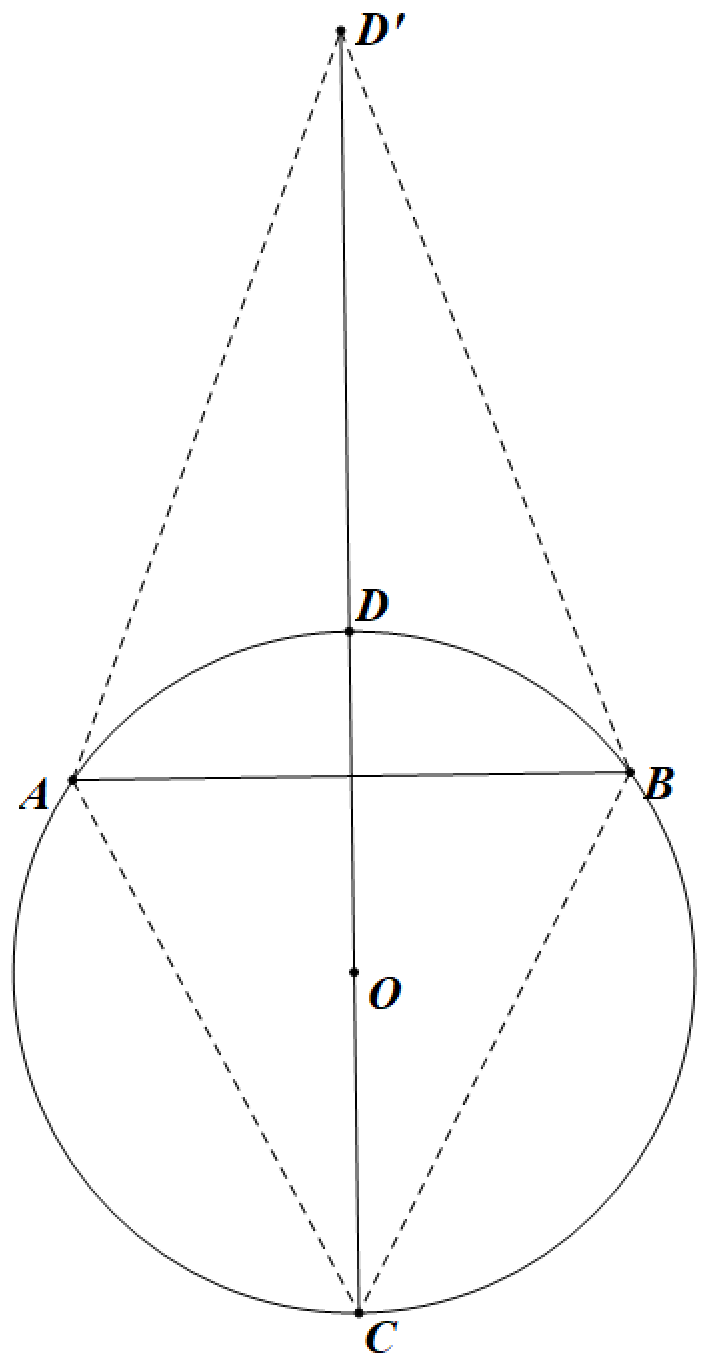}
    \end{minipage}
    }
\end{figure}

\begin{remark}
On the other hand, in the theorem above, the condition that the Menger curvatures of $A,P_1, P_2 \cdots P_n, B$ satisfy $M\ge 1/R$ can not be changed to $M\le 1/R$. Still consider  $C(0,-1),D(0,1),A,B \in S^1$ with $y_A=y_B$ in Figure~5, then we move $D$ to $D'(0, y_0')$, s.t. $y_0' > 1$ and $sin\angle AD'B < sin\angle ADB$. It is also easy to vertify that $M(A),M(D'),M(B) <1$ in the polygon $ACBD'$. This seems different from the smooth category. Yet upon close examination you will find this is completely analogous to the bumbbell-shaped deformation we mentioned in Subsection~6.2; see the pictures there and Remark~\ref{remark-dumbbell}. That means that such a nontrivial perturbation exist only when the perturbation is \emph{large enough}; otherwise the rigidity still holds true as in Theorem~\ref{thm-rigid-4}.
\end{remark}

\par
\noindent
\textbf{Acknowledgement.} \par
We would like to thank Baichuan Hu, the classmate of the third author, who suggested the basic idea of the proof in the Appendix. The valuable suggestion of Ying Lv helps to simplify this proof significantly.

\end{document}